\documentclass[11pt]{amsart}
\usepackage{geometry}                
\usepackage{graphicx}
\usepackage{amssymb}
\usepackage{epstopdf}
\usepackage{amsmath}
\usepackage{color}
\usepackage{hyperref}
 \usepackage{tikz}
 
\usepackage[all]{xy}
\xyoption{matrix}
\xyoption{arrow}
 
 \usetikzlibrary{arrows,decorations.pathmorphing,backgrounds,positioning,fit,petri}

 \tikzset{help lines/.style={step=#1cm,very thin, color=gray},
help lines/.default=.5} 
\tikzset{thick grid/.style={step=#1cm,thick, color=gray},
thick grid/.default=1} 

\newtheorem{thm}{Theorem}[section]
\newtheorem{lem}[thm]{Lemma}
\newtheorem{cor}[thm]{Corollary}
\newtheorem{prop}[thm]{Proposition}

\theoremstyle{definition}
\newtheorem{defn}[thm]{Definition}
\newtheorem{eg}[thm]{Example}

\newtheorem{rem}[thm]{Remark}
 
\numberwithin{equation}{section}

\newcommand\into{\hookrightarrow}

\newcommand\onto{\twoheadrightarrow}

\newcommand\smallcoprod{\,{\textstyle{\coprod}}\,}
\newcommand\smallprod{\,{\textstyle{\prod}}\,}

\DeclareMathOperator{\Hom}{Hom}%
\DeclareMathOperator{\Ext}{Ext}%
\DeclareMathOperator{\add}{add} 

\newcommand{\field}[1]{\mathbb{#1}}
\newcommand{\ZZ}{\ensuremath{{\field{Z}}}}

\newcommand{\RR}{\ensuremath{{\field{R}}}}

\newcommand{\NN}{\ensuremath{{\field{N}}}}

\newcommand{\arUD}[1]{$\nearrow$\put(.02,.15){$#1$}\put(.2,0){$\searrow$}}
\newcommand{\arDU}[1]{$\searrow$\put(.02,-.15){$#1$}\put(.2,0){$\nearrow$}}

\newcommand{\commentout}[1]{}

\newcommand{\cC}{\ensuremath{{\mathcal{C}}}}

\newcommand{\cF}{\ensuremath{{\mathcal{F}}}}

\newcommand{\cJ}{\ensuremath{{\mathcal{J}}}}

\newcommand{\cP}{\ensuremath{{\mathcal{P}}}}

\newcommand{\cS}{\ensuremath{{\mathcal{S}}}}

\newcommand{\cT}{\ensuremath{{\mathcal{T}}}}

\newcommand{\cluster}[3]{
\draw[very thick, rotate=#1] (#2,0) circle [radius=#3cm]}
\newcommand{\dotcluster}[3]{
\draw[very thick,dotted, rotate=#1] (#2,0) circle [radius=#3cm]}

\newcommand{\lcluster}[3]{
\draw[rotate=#1] (#2,0) circle [radius=#3cm]}
\newcommand{\dcluster}[3]{
\draw[thick, dashed,rotate=#1] (#2,0) circle [radius=#3cm]}
 \newcommand{\bcluster}[3]{
\draw[color=blue,very thick, rotate=#1] (#2,0) circle [radius=#3cm]}
\newcommand{\rcluster}[3]{
\draw[color=red,very thick, rotate=#1] (#2,0) circle [radius=#3cm]}
\newcommand{\ncluster}[3]{
\draw[very thick, rotate=#1] (-#2,0) circle [radius=#3cm]}
\newcommand{\lncluster}[3]{
\draw[rotate=#1] (-#2,0) circle [radius=#3cm]}
\newcommand{\bncluster}[3]{
\draw[color=blue,very thick, rotate=#1] (-#2,0) circle [radius=#3cm]}
\newcommand{\rncluster}[3]{
\draw[color=red,very thick, rotate=#1] (-#2,0) circle [radius=#3cm]}

\title{Cyclic posets and triangulation clusters}
\author{Kiyoshi Igusa}
\address{Department of Mathematics, Brandeis University, Waltham, MA 02454}\email{igusa@brandeis.edu}

\author{Gordana Todorov}
\address{Department of Mathematics, Northeastern University, Boston, MA 02115}
\email{g.todorov@northeastern.edu}



\begin{document}

\begin{abstract} Triangulated categories coming from cyclic posets were originally introduced in \cite{IT15b} as a generalization of the constructions of various triangulated categories with cluster structures. We give an overview, then analyze ``triangulation clusters'' which are those corresponding to topological triangulations of the 2-disk.
Locally finite nontriangulation clusters give topological triangulations of the ``cactus space'' associated to the ``cactus cyclic poset''.
\end{abstract}

\maketitle

\tableofcontents

\section{Introduction}
We first point out that in our papers cyclic posets are not posets. However, to each cyclic poset $X$ there is an associated ``covering" poset $\widetilde X$. When $(X,c)$ is nondegenerate, the covering poset $\widetilde X$ is also a $\mathbb Z$-poset in the sense of \cite{ZZK} and \cite{Sh}. Then the elements of $X$ are in bijection with the cyclic $\mathbb Z$-subposets of $\widetilde X$. For the precise statements, 
see Section \ref{ss: cyclic S-posets} and Theorems \ref{thm: cyclic poset is set of cyclic Z-posets}, \ref{prop: converse to cyclic poset is set of cyclic Z-posets}. When we introduced the notion of cyclic poset in \cite{IT15b} we were not aware of the existing notion of cyclic $\mathcal S$-poset. These are different notions, as stated above. There is another related notion called a ``partial cyclical order'' \cite{M} which appears as a special case of a cyclic poset. See Section \ref{ss: PCO}. We thank the referee for making us aware of these other notions. 

This paper has one section giving an overview about cyclic posets which is an expanded version of the lectures given at Tsinghua University and Chern Institute at Nankai University.  The rest of the paper is material which was not published before.

Using representations of cyclic posets we give  a general approach to the construction of {\color{black}several cluster categories such as the} continuous cluster categories \cite{IT15a}, cluster categories of types $A_n$ (\cite{CCS}, \cite{BMRRT}), {\color{black}  $A_{\infty}$ \cite{HJ12}, $m$-cluster categories of type $A_{\infty}$ \cite{HJ15} and many more, for example \cite{LP}, \cite{GHJ}}. These categories were also applied in the recent study of infinite Borel subalgebras of $\mathfrak{sl}(\infty)$ \cite{JY}.

The main results of this paper are about the cluster structures on the triangulated categories arising from  the subsets of circle $S^1$ together with several types of automorphisms. 
When the automorphism is identity and the subset is the entire $S^1$, we obtain continuous cluster category, which was already describe in several papers \cite{IT15a} and\cite{IT15b}. In this case we only make a quick summary and some remarks about differences with classical cluster categories. With the same identity automorphsim, one obtains `spaced-out' cluster categories, whose clusters are in bijection with the triangulations of the $n$-gon, however they are different from the classical cluster categories of type $A_{n-3}$ whose clusters are also in bijections with the triangulations of the $n$-gon. For example, in the space-out cluster category every object is isomorphic to its shift, i.e. $X\cong X[1]$, whereas the cluster category of type $A_{n-3}$ is 2-Calabi-Yau. The spaced-out cluster category of the $n$-gon embeds as a full subcategory of the standard cluster category of $A_{2n-3}$ and the embedding takes clusters to clusters.

We consider discrete subsets of $S^1$ and the canonical automorphism. Many new triangulated categories with cluster structrues are obtained this way. {\color{black}We are particularly interested in the clusters corresponding to topological triangulations of subsets of the disk $D^2$. We show (Lemma \ref{1-1}) that these correspond to the locally finite clusters.} Special attention is paid to  {\color{black} a particularly nice collection of locally finite} clusters which define geometric triangulations of the disc $D^2$ after the limit points of the subset $Z$ are removed. We call these `triangulation clusters'. We give a complete characterization of triangulation clusters and we also consider the more difficult question of the topology of locally finite clusters which are not triangulation clusters. To handle these cases, we construct a new cluster category called the `cactus category' using  a new kind of cyclic posets called the `cactus cyclic posets'. We show that the nontriangulation clusters are in bijection with triangulation clusters for the components of the `cactus category' which is isomorphic to a product of cluster categories corresponding to `smaller disks'. Thus the nontriangulation clusters also are in bijection with geometric triangulations of a `cactus space'.

In the case when $\varphi$ is canonical and the cyclic poset has only finitely many limit points, we also characterize the shape of any locally finite cluster in the Auslander-Reiten quiver. In the case when there are two limit points the cluster starts with any zig-zag in the $A_\infty^\infty$ component and pushes down portions to the two $A_\infty$ components.

Finally we consider triangulated categories given by arbitrary subsets $X$ of the circle which are invariant under rotation by an angle $\theta$. We determine exactly when the resulting triangulated category $\cC_{\varphi_\theta}(X)$ has a cluster structure and, also, what all possible cluster structures are on these sets. When $X$ is infinite, the exchange graph of $\cC_{\varphi_\theta}(X)$ is infinite but each component is finite. So, only finitely many clusters are reachable from any seed. Also, clusters are almost never maximal compatible sets. So, the clusters are not defined in this way.

The original motivation for this paper was to present a few unpublished examples of cluster categories coming from cyclic posets. However, the paper now has several new results as outlined above. We also show the existence of triangulation clusters in cases where there are infinitely many limit points to the cyclic poset (and even infinitely many limits of limit points) showing that there are nontrivial examples in these extreme cases. It would be interesting to examine quantized versions of these infinite cluster structures.

The authors would like to thank Professors Bin Zhu, Fang Li, Zhongzhu Lin for their hospitality at Tsinghua University and at the Chern Institute of Mathematics during the Workshop on Cluster Algebras and Related Topics, July 10-13, 2017.
A series of lectures on this topic was given by the authors, which motivated the beginning of this paper and subsequent work for the rest of the paper.

\section{Representations of Cyclic Posets} \label{sec 1}

In this section we recall the definition and basic properties of cyclic posets from the paper \cite{IT15b}. 
As stated at the beginning, 
the notion of cyclic poset that is used in our papers, is different from the notion of cyclic $\mathcal S$-poset of \cite{ZZK} or \cite{Sh}  
(see Section \ref{ss: cyclic S-posets}). The notion of cyclic poset is also different, but related to the notion of \emph{partial cyclic order} \cite{M}. (See Section \ref{ss: PCO} for precise statements.)

\subsection{Cyclic Posets $(X,c)$} We first recall the definitions which will be used in the paper.

\begin{defn}\label{def: cocycle} A \emph{3-cocycle} on a set $X$ is defined to be a mapping $c: X^3\to \mathbb N$ with \emph{coboundary} $\delta c:X^4\to \ZZ$ equal to zero where
\[
	\delta c(w,x,y,z):=c(x,y,z)-c(w,y,z)+c(w,x,z)-c(w,x,y).
\]
The cocycle is called \emph{reduced} if $c(x,x,y)=0$ and $c(x,y,y)=0$ for all $x,y\in X$. The cocycle $c$ is \emph{degenerate} if $c(x,y,x)=0$ for some $x\neq y\in X$. If no such pair exists, we say $c$ is \emph{nondegenerate}.
\end{defn}

\begin{defn}\label{def: cyclic poset}
A \emph{cyclic poset} is a pair $(X,c)$ where $c$ is a reduced 3-cocycle on $X$. The cyclic poset is \emph{degenerate} or \emph{nondegenerate} if $c$ is \emph{degenerate} or \emph{nondegenerate}, respectively.
\end{defn}

\begin{rem} \label{subset} Let $(X,c)$ be a cyclic poset. Let $X'\subset X$ be a subset and $c'=c|_{X'}$. Then $(X',c')$ is a cyclic poset. This will be used in 
{\color{black} Sections \ref{subsets}, \ref{phi=canonical} }where cluster structures on cyclic posets $(X,c)$ with admissible subsets $X\subset S^1$ will be described.
\end{rem}
The main property of cyclic posets, which is used extensively in the original paper \cite{IT15b} and here as well, is the fact that, although cyclic posets are not posets, they correspond to actual posets which we call ``covering posets'', which we now describe. Here and in the original paper \cite{IT15b} a \emph{poset} is defined to be a set with a reflexive, transitive relation $x\le y$ and a \emph{poset morphism} is one which preserves this relation. The notation $x<y$ means $x\le y$ and $y\not\le x$.

\begin{defn}\label{cover} (Definition 1.1.1.\cite{IT15b})
A  \emph{covering poset} of a set $X$ is defined to be a poset $\widetilde X$ and a set map $\pi:\widetilde X\to X$ which satisfy the following conditions:
\begin{enumerate}
\item $\forall x\in X$ there is a poset isomorphism $\pi^{-1}(x)\cong \ZZ$.
\item The map $\sigma\!:\!\widetilde X\to\widetilde X$ given by $\sigma(\tilde x)\!:=\!
(\text{smallest element in }\pi^{-1}\pi(\tilde x)$ so that $\sigma(\tilde x)\!>\!\tilde x$) is a poset isomorphism of $\widetilde X$ which we call the \emph{defining automorphism}.
\item $\forall  \tilde x,\tilde y\in \widetilde X$ there exists $m,n\in \mathbb N$ such that $\tilde x\le \sigma^m\tilde y\le\sigma^{m+n}\tilde x$.
\end{enumerate}
Define 
$\mathcal P\!_{cover}(X)$ to be the set of all covering posets $\pi:\widetilde X\to X$. 
\end{defn}

\begin{rem}
On each fiber $\pi^{-1}(x)$ of $\pi:\widetilde X\to X$, the defining automorphism $\sigma$ corresponds to addition of 1 in $\ZZ\cong \pi^{-1}(x)$, i.e., $\sigma(\tilde x)=\lambda^{-1}(1+\lambda(\tilde x))$ for any choice of poset isomorphisms $\lambda: \pi^{-1}(x)\cong\ZZ$.
\end{rem}

\begin{prop}[Corollary 1.1.11, \cite{IT15a}]\label{equivalence: cyclic poset=covering poset} Let $X$ be a set. Then there exists a bijection between the following sets $\cP_{cycl}(X)\cong \cP_{cover}(X)$ where:
$$\mathcal P\!_{cycl}(X)=\{\text{cyclic posets } (X,c) \},$$ 
$$\mathcal P\!_{cover}(X)=\{\pi:\widetilde X\to X\ |\ \text{covering posets}\text{ satisfying (1), (2), (3)} \text{ from Definition } \ref{cover}\}.$$
Furthermore, a cyclic poset $(X,c)$ is nondegenerate if and only if the corresponding covering poset $\widetilde X$ has antisymmetric partial ordering.
\end{prop}

\begin{proof} The bijection $\mathcal P\!_{cover}(X) \to \mathcal P\!_{cycl}(X)$ is given as follows: 
let $\pi:\widetilde X\to X$ be a covering poset. Let $\lambda: X\to \widetilde X$ be a section.
For each pair $x,y\in X$, define $b(x,y):=(\text{the smallest } n\in \mathbb Z \ |\ \lambda(x)\le\sigma^n\lambda(y))$.
Then it is shown in Lemma 1.1.8 in [IT15b] that $c:X^3\to \mathbb Z$ defined as $c(x,y,z):=b(x,y)+b(y,z)-b(x,z)$  is a reduced cocycle and Theorem 1.1.10 in \cite{IT15b} shows that the correspondence given by $(\pi:\widetilde X\to X)\mapsto (X,c)$ is a bijection.

To prove the last statement, suppose that $(X,c)$ is degenerate. Then there exist $x\neq y$ in $X$ with $c(x,y,x)=b(x,y)+b(y,x)=0$. Let $n=b(x,y)$. Then $-n=b(y,x)$. Then $\lambda(x)\le \sigma^n\lambda(y)\le \lambda(x)$ but $\lambda(x)\neq \sigma^n\lambda(y)$. So, these two elements of $\widetilde X$ show that the relation in $\widetilde X$ is not antisymmetric. The converse holds by a similar calculation. 
\end{proof}

\begin{eg}\label{cocycle10}Let $X=S^1$, $\widetilde X=\mathbb R$. Let $c:X^3\to N$ be given by:\\
$c(x,y,z)= 1$ if $\angle (x,y)+\angle (y,z) \geq 2\pi$, i.e. going counterclockwise around the circle the ordered elements $x,y,z$ complete the circle,\\
$c(x,y,z)= 0$ if $\angle (x,y)+\angle (y,z) < 2\pi$, i.e. going counterclockwise around the circle the ordered elements $x,y,z$ do not completes the circle.
\end{eg}

\subsection{Partial cyclic order}\label{ss: PCO} 
We consider the relation between cyclic posets (Definition \ref{def: cyclic poset}) and partial cyclic orders (Definition \ref{def: PCO}). We show that certain cyclic posets on a set $X$ define partial cyclic orders on the set $X$, however not all partial cyclic orders on $X$ can be obtained in this way as we show in Remark \ref{rem: eg: not all PCOs}. Hence, the two notions are different in general, but agree in some cases.

\begin{defn}\label{def: PCO}
A \emph{partial cyclic order} on a set $X$ is defined to be a set $\Delta$ of ordered triples of distinct elements of $X$ having the following properties.
\begin{enumerate}
\item[(a)] If $(x,y,z)\in \Delta$ then $(y,z,x), (z,x,y)\in \Delta$.
\item[(b)] If $(x,y,z)\in \Delta$ then $(x,z,y)\notin\Delta$.
\item[(c)] If $(x,y,z),(x,z,w)\in \Delta$ then $(x,y,w)\in \Delta$.
\end{enumerate}
A partial cyclic order is \emph{complete} if it satisfies the additional property that, for any three distinct elements $x,y,z\in X$, either $(x,y,z)\in \Delta$ or $(x,z,y)\in\Delta$.
\end{defn}

\begin{prop}\label{prop: PCO as case of cyclic poset}
Let $(X,c)$ be a cyclic poset satisfying the following for some $r\ge2$.
\begin{enumerate}
\item $c(x,y,z)\le r$ for all $x,y,z\in X$.
\item $c(x,y,x)=r$ when $x\neq y$.
\end{enumerate}
Then we have the following for all $x,y,z$ distinct elements of $X$.
\begin{enumerate}
\item[(i)] $c(x,y,z)+c(x,z,y)=r$.
\item[(ii)] $c(x,y,z),c(x,z,y)<r$ imply $c(x,y,z),c(x,z,y)\gneq 0$.
\end{enumerate}
Let $\Delta$ be the set of all triples $(x,y,z)$ of distinct element of $X$ so that $c(x,y,z)=r$. Then $\Delta$ is a partial cyclic order on $X$.
\end{prop}

\begin{proof}
We first verify (i) and (ii). Then we will verify the properties in Definition \ref{def: PCO}.
\begin{enumerate}
\item[(i)] For distinct elements $x,y,z\in X$ we have $\delta c(x,z,y,z)=c(z,y,z)-c(x,y,z)+c(x,z,z)-c(x,z,y)=0$. Since $c(z,y,z)=r$ and $c(x,z,z)=0$ we conclude that $c(x,y,z)+c(x,z,y)=r$ are claimed. 
\item[(ii)] This follows immediately from (i).
\item[(a)] $c(x,y,z)=c(y,z,x)$ for all $x,y,z$ distinct since $\delta c(x,y,z,x)=c(y,z,x)-c(x,z,x)+c(x,y,x)-c(x,y,z)=0$ and $c(x,z,x)=c(x,y,x)=r$. Thus $(x,y,z)\in \Delta$ if and only if $(y,z,x)\in \Delta$.
\item[(b)] If $(x,y,z)\in\Delta$ then $c(x,z,y)=0$ by $(i)$. Therefore $(x,z,y)$ is not in $\Delta$. 
\item[(c)] If $(x,y,z),(x,z,w)\in \Delta$ then $c(x,y,w)+c(y,z,w)=c(x,y,z)+c(x,z,w)=2r$. So, we must have $c(x,y,w)=c(y,z,w)=r$. So, $(x,y,w)\in\Delta$.
\end{enumerate}
Thus, $\Delta=\{(x,y,z)\,|\, c(x,y,z)=r\}$ is a partial cyclic order on $X$.
\end{proof}

\begin{rem}\label{rem: eg: not all PCOs} We observe that there are partial cyclic orders which cannot be defined as triples with constant cocycle value $r$ for a bounded cocycles as in Proposition \ref{prop: PCO as case of cyclic poset}. For example, let $X=\NN=\{0,1,2,\cdots\}$ with $\Delta$ given by:
\[
	\Delta:=\{(x,y,z)\in\NN^3\,|\, (x-y)(y-z)(z-x)>0\}
\]
Suppose there is a reduced cocycle $c$ on $X=\NN$ satisfying the conditions of Proposition \ref{prop: PCO as case of cyclic poset} so that $c(x,y,z)=r$ if and only if $(x,y,z)\in\Delta$. Then we will obtain a contradiction. For any $n\ge2$, $(1,n,n+1)\in\Delta$. So $c(1,n,n+1)=r$. Using the fact that $c(x,y,z)+c(x,z,y)=r$ for all $x,y,z$ distinct we conclude that $c(1,n+1,n)=0$. Since $c(1,n+1,n)$ is the first term in the expansion of $\delta c(0,1,n+1,n)$, we get
\[
	\delta c(0,1,n+1,n)=-c(0,n+1,n)+c(0,1,n)-c(0,1,n+1)=0.
\]
But $c(0,n+1,n)>0$ by \ref{prop: PCO as case of cyclic poset}(ii) since neither $(0,n+1,n)$ nor $(0,n,n+1)$ lies in $\Delta$. So,
\[
	c(0,1,n)>c(0,1,n+1)
\]
for all $n\ge2$. So, $c(0,1,n)$ is a monotonically decreasing function of $n$. Since $c(0,1,2)<r$, after $r$ steps we obtain $c(0,1,r+2)\le c(0,1,2)-r<0$, contradicting the fact that $c(x,y,z)\ge0$ for all $x,y,z$.
\end{rem}

\subsection{Cyclic $\mathbb Z$-posets}\label{ss: cyclic S-posets}
We discuss the relation between cyclic posets and cyclic $\mathbb Z$-posets which are a special case of cyclic $\mathcal S$-posets, in the sense of \cite{ZZK}, \cite{Sh} as we recall now.

\begin{defn}\label{relation} A partially ordered monoid, also called a \emph{pomonoid}, is a monoid $\mathcal S$ with a reflexive, antisymmetric and transitive relation $x\le y$, so that $a\le b$ and $x\le y$ implies $ax\le by$. An \emph{$\mathcal S$-poset} is defined to be a set with reflexive, antisymmetric and transitive relation with an action of $\mathcal S$ which is compatible with the partial ordering relations on both sets. Such an $\mathcal S$-poset  is called a \emph{cyclic $\mathcal S$-poset} if it is generated by one element, i.e., if it is equal to $\mathcal S a$ for some $a$.
\end{defn}

When the pomonoid $\mathcal S$ is the ordered additive group $\mathbb Z$, then, in many cases, the covering poset $\widetilde X$ from Definition \ref{cover} is an example of a $\mathbb Z$-poset and the elements of the set $X$ are in bijection with the cyclic $\mathbb Z$-subposets of $\widetilde X$. The precise statement is given below.

\begin{thm}\label{thm: cyclic poset is set of cyclic Z-posets}
Let $(X,c)$ be a nondegenerate cyclic poset with covering poset $\pi:\widetilde X\to X$. 

\begin{enumerate}
\item[(a)] Then $\widetilde X$ is a $\mathbb Z$-poset with the action of $\mathbb Z$ given by $n\cdot \tilde x:=\sigma^n\tilde x$ for $n\in \mathbb Z$. 
\item[(b)] The mapping which sends $x\in X$ to $\pi^{-1}(x)\subseteq \widetilde X$ is a bijection between $X$ and the set of cyclic $\mathbb Z$-subposets of $\widetilde X$.
\end{enumerate}
\end{thm}

\begin{proof}
(a) Since $\sigma$ is a poset automorphism of $\widetilde X$ and $\tilde x<\sigma \tilde x$ for all $\tilde x\in \widetilde X$, we get $\sigma^n\tilde x<\sigma^m\tilde x$ whenever $n<m$. So, the action of $\mathbb Z$ on $\widetilde X$ given by $n\cdot \tilde x:=\sigma^n \tilde x$ is order preserving in both variables $n\in \mathbb Z$ and $\tilde x\in\widetilde X$. By Theorem \ref{equivalence: cyclic poset=covering poset}, the relation on $\widetilde X$ is antisymmetric. Therefore, $\widetilde X$ is a $\mathbb Z$-poset. 

(b) By Definition \ref{cover}(1), each inverse image $\pi^{-1}(x)$ is a single $\mathbb Z$-orbit in $\widetilde X$ which is the same as a cyclic $\mathbb Z$-subposet of $\widetilde X$. Conversely, every cyclic $\mathbb Z$-subposet of $\widetilde X$ is, by definition, of the form $\mathbb Z\cdot \tilde x=\sigma^{\mathbb Z}\tilde x=\pi^{-1}\pi(\tilde x)$ which is one fiber of the map $\pi:\widetilde X\to X$. Therefore, $\pi$ maps each cyclic $\mathbb Z$-subposet to an element of $X$ and $\pi^{-1}$ is the inverse of that map. So, we have a bijection as claimed.
\end{proof}

We now describe which $\mathbb Z$-posets occur as covering posets of nondegenerate cyclic posets.

\begin{prop}\label{prop: converse to cyclic poset is set of cyclic Z-posets}
Let $P$ be a $\mathbb Z$-poset. Then $P$ is isomorphic as $\mathbb Z$-poset to the covering poset $\widetilde X$ of some nondegenerate cyclic poset $(X,c)$ if any only if it has the following property:

$(\ast)$ For any $a,b\in P$ there is an $n\in\mathbb Z$ so that $a<n\cdot b$ and $a\not<(n-1)\cdot b$.
\end{prop}

\begin{proof}
Suppose that $P$ is a $\mathbb Z$-poset with property $(\ast)$. Let $X=P/\mathbb Z$ be the set of orbits of the action of $\mathbb Z$ on $P$ and let $\pi:P\to X$ be the projection map. We shall verify that $P$ satisfies the three conditions in Definition \ref{cover} to show it is a covering poset of $X$.
\begin{enumerate}
\item Putting $a=b$ in $(\ast)$ we see that the action of $\mathbb Z$ on $P$ is a free action. Since the action is order preserving, each $\mathbb Z$-orbit must be poset isomorphic to $\mathbb Z$ as required.
\item Given $p\in P$, $\pi^{-1}\pi(p)$ is equal to the $\mathbb Z$-orbit of $p$. Thus the smallest element of $\pi^{-1}\pi(p)$ greater than $p$ is $1\cdot p$. Therefore, the action of $\sigma$ on $P$ is the action of $1\in\mathbb Z$. Since $1$ has inverse $-1\in \mathbb Z$, this action is invertible and, therefore, $\sigma$ is a poset automorphism of $P$ as required.
\item Condition (3) states: $\forall a,b\in P$ there exist $n,m\in\NN$ so that $a\le \sigma^n b\le \sigma^{n+m}a$.

Since $\sigma^na< \sigma^ma$ for all $n<m$, it suffices to find integers $n,m$ satisfying the inequalities above.
This follows from $(\ast)$ for $a,b$ and for $b,a$. Then there are $n,m\in \mathbb Z$ so that $a<n\cdot b=\sigma^n(b)$ and $b<m\cdot a=\sigma^m(a)$. This proves condition (3).
\end{enumerate}
Therefore, $P$ is a covering poset of the set $X=P/\mathbb Z$. The corresponding cyclic poset $(X,c)$ is nondegenerate by Theorem \ref{equivalence: cyclic poset=covering poset}.

Conversely, let $(X,c)$ be a nondegenerate cyclic poset. Then the covering poset $\widetilde X$ is a $\mathbb Z$-poset by Theorem \ref{thm: cyclic poset is set of cyclic Z-posets} and it satisfies $(\ast)$ by Definition \ref{cover}(3).
\end{proof}

\begin{eg}\label{eg: Zn}
For any integer $n\ge 2$, let $\frac1n\ZZ$ be the set of all rational numbers $z$ so that $nz$ is an integer. This is an $\mathcal S$-poset with $\mathcal S=\ZZ$ with the action of $m\in \ZZ$ on $\frac1n\ZZ$ given by adding $m$. Then the cyclic $\ZZ$-subposets of $\frac1n\ZZ$ are the $n$ cosets $[z]:=z+\ZZ$ of $\ZZ$ in $\frac1n\ZZ$. These correspond to the $n$ elements of the following cyclic poset $(Z_n,c)$ with cocycle $c$ given in Example \ref{cocycle10} and
\[
	Z_n:=\{e^{2\pi iz}\in S^1\,|\, z\in \tfrac1n\ZZ\}.
\]
\end{eg}

\subsection{Frobenius Categories $\mathcal FX$ associated to $(X,c)$}
The Frobenius categories of certain representations associated to cyclic posets that we will consider, will be $t^{\mathbb N}$-categories as defined by \cite{D} and \cite{vR}, which we now recall.  
\begin{defn} Let $k$ be a field and let $R=k[[t]]$. Then a $t^{\mathbb N}$-category over a set $X$, denoted by $RX$, is defined as a category which has:
\begin{itemize}
\item Indecomposable objects: to each $x\in X$ correspond an object denoted by $P_x$.
\item Morphisms: Hom$_{RX}(P_x,P_y)=Rf_{xy}\cong R$.
\item Compositions: $(sf_{yz})(rf_{xy})=(sr) t^nf_{xz}$ for some $n\in\mathbb N$.
\end{itemize}
\end{defn}

\begin{prop}[Proposition 1.2.5 in \cite{IT15a}] Let $(X,c)$ be a cyclic poset. Then there exists a $t^{\mathbb N}$-category over $X$ with composition given by 
 $(sf_{yz})(rf_{xy})=(sr) t^{c(x,y,z)}f_{xz}$ for all $s,r\in R$.
\end{prop}

Given a cyclic poset $(X,c)$ and associated $t^{\mathbb N}$-category $RX$, we define the following category $\mathcal FX$, factorization category, which is shown to be Frobenius and consequently, its stable category will be triangulated.

\begin{defn} Let $(X,c)$ be a cyclic poset and $RX$ the associated $t^{\mathbb N}$-category. Define the \emph{factorization category} $\mathcal FX$ as the additive category with:
\begin{itemize}
\item Indecomposable objects: $E(x,y):=\left(P_x\oplus P_y, \varphi=
\begin {bmatrix}0&\beta \\ \alpha&0\end{bmatrix},
\varphi^2 = t \begin {bmatrix}1&0\\ 0&1\end{bmatrix}\right)$.
\item Morhisms: $f: E(x,y)\to E(x',y')$ so that $f\varphi = \varphi' f$.
\end{itemize}
\end{defn}

\begin{rem} \label{geodesic}
When the set $X$ is a subset of the circle $S^1$, any two elements $x,y\in X$ are points on the circle $S^1$ and, when $x\neq y$, the object $E(x,y)$ corresponds to the geodesic connecting $x,y$. 
\end{rem}

\begin{rem}\label{crossing}\label{geodesics} We will be using term `geodesic' in two different situations.
In Section \ref{CCC} the objects of the continuous cluster category $\cC_{cont}$ will correspond to actual geodesics in the hyperbolic plane $\mathfrak h^2$.
In other situations, when  the set $X$ is a subset of the circle $S^1$, any two elements $x,y\in X$ are points on the circle $S^1$. In that case we will associate to the object $E(x,y)$  the `closed geodesic' connecting $x,y$, i.e. geodesic of the hyperbolic plane together with the points $x, y$. We use this correspondence to draw collections of objects in various examples.
\end{rem}
\begin{rem} \label{noncrossing}
Two geodesics are said to be \emph{noncrossing} if they do not intersects. Two closed geodesics are called noncrossing if they do not intersect on their interiors.
\end{rem}
\begin{rem} \label{compatible/noncrossing}
When the indecomposable objects of a triangulated category correspond to geodesics, there is often a notion of ``compatibility'' of such objects when the corresponding geodesics do not cross. In Section 2 an algebraic definition of compatibility is given for objects $E(x,y)$ and $E(x',y')$ and this happens precisely when the corresponding geodesics are noncrossing:
For the continuous cluster category $\cC_{cont}$  and categories with spaced-out clusters, this is proved in Proposition 4.1.3 in \cite{IT15b}. In Section \ref{phi=canonical}, the same statement holds for triangulated categories  $\cC_{can}(Z)$ associated to admissible subsets $Z$ of $S^1$ (Lemma 2.4.4 in \cite{IT15a}).
\end{rem}
In order to construct triangulated categories, we will first construct Frobenius categories; recall that a category is Frobenius if: 1) it is an exact category, 2) projective and injective objects coincide and 3) it has enough projectives and it has enough injectives (\cite{HappelBook}, p.11).

\begin{thm}[Theorem 1.4.7, \cite{IT15b}] Let $\mathcal FX$ be the factorization category associated to the cyclic poset $(X,c)$. Then the category $\mathcal FX$ is Frobenius where: 
\begin{enumerate}
\item The indecomposable projective-injective objects in $\mathcal FX$ are (up to isomorphism) given by 
$E(x,x)=\left(P_x\oplus P_x, \varphi=
\begin {bmatrix}0&t \\ 1&0\end{bmatrix}\right)$, and 
\item For every indecomposable object $E(x,y)$, there exists a short exact sequence:
$$0\rightarrow E(x,y)\xrightarrow{i}E(x,x)\smallcoprod E(y,y)\xrightarrow{p}\Sigma E(x,y)\rightarrow 0.$$
Here $i$ is injective envelope, $p$ is projective cover and $\Sigma E(x,y)$ is the cokernel of $i$. 
\end{enumerate}
\end{thm}
 The notation $\Sigma$ is used since the image of $\Sigma E(x,y)$ in the stable (triangulated) category is the shift of the image of $E(x,y)$ as stated  below in the Corollary \ref{triangulated}.
\noindent Since $\mathcal FX$ is a Frobenius category, the stable category $\underline{\mathcal FX}$ will be triangulated  by the theorem of Happel \cite{HappelBook}. So, we have the following corollary.

\begin{cor} \label{triangulated} Let $\underline{\mathcal FX}$ be the stable category of the Frobenius category $\mathcal FX$ associated to the cyclic poset $(X,c)$. 
Then $\underline{\mathcal FX}$  is a triangulated category and distinguished triangles 
$E(x,y)\xrightarrow{f} E(x',y')\xrightarrow{i'}E\rightarrow \Sigma E(x,y)$
can be obtained from the pushout diagrams of the exact sequences
$0\rightarrow E(x,y)\xrightarrow{i}E(x,x)\smallcoprod E(y,y)\xrightarrow{p}\Sigma E(x,y)\rightarrow 0$
along the maps $E(x,y)\xrightarrow{f}E(x',y')$.
\end{cor}

\begin{eg} \label{ccc} Let $(X,c)$ be the cyclic poset from Example \ref{cocycle10}, i.e. $X=S^1$ and the cocycle $c(x,y,z)= 1$ if ordered triple $x,y,z\in S^1$ completes the circle and $c(x,y,z)= 0$ if ordered triple $x,y,z\in S^1$ does not complete the circle.
 Indecomposable objects $E(x,y)$ correspond to the arcs between the points $x,y\in S^1$.
 The objects $E(x,x)$ are projective-injective.
 The stable category $\underline{\mathcal FX}$ is the continuous cluster category $\mathcal C_{cont}$ of \cite{IT15b}. More details will be given in Section \ref{CCC}.
\end{eg}

\subsection{Cyclic posets with automorphisms - general twisted version: $(X,c)_{\varphi}$, $\mathcal F_{\varphi}(X)$} \label{Cyclic posets with automorphisms}

 We recall definitions and basic properties of admissible automorphisms from \cite{IT15a}. With  admissible automorphisms of cyclic posets we construct new families of Frobenius categories. This construction creates  triangulated categories again, and produces some new classes of categories with cluster structures. 
 
\begin{defn} \label{admissible automorphism} 
Let $(X,c)$ be a cyclic poset
with $\sigma: \widetilde X\to \widetilde X$ as in Definition \ref{cover}. An automorphism $\varphi$ of $(X,c)$ is called
\emph{admissible} if there is a $\sigma$-equivariant poset automorphism $\tilde\varphi$ of $\widetilde X$ which \emph{covers} $\varphi$ in the sense that $\pi \tilde \varphi=\varphi \pi:\widetilde X\to X$ and satisfies $\tilde x\le \tilde\varphi(\tilde x)\le \tilde\varphi^2(\tilde x)<\sigma\tilde x$ for all $\tilde x\in \widetilde X$. 
We denote by $(X,c)_{\varphi}$ a cyclic poset with admissible automorphism $\varphi$.
\end{defn}

\begin{eg}\label{eg: Zn and phi}
A basic example is $X=Z_n$ from Example \ref{eg: Zn}  for $n\ge 3$ with $\widetilde X=\frac1n\ZZ$ and $\pi^{-1}\pi(\tilde z)=\tilde z+\ZZ$. The \emph{standard admissible automorphism} $\tilde \varphi$ is given by $\tilde\varphi(\tilde z)=\tilde z+\frac1n$ (the smallest element of $\frac1n\ZZ$ larger than $\tilde z$). Then $\tilde\varphi^2(\tilde z)=\tilde z+\frac2n<\sigma(\tilde z)=\tilde z+1$ for all $\tilde z\in \frac1n\ZZ$ since $n\ge3$.
\end{eg}

The cyclic poset automorphism $\varphi$ induces an automorphism of the associated $t^{\mathbb N}$-category $RX$ giving the natural morphisms $P_x\xrightarrow{\eta_x}\varphi P_x=P_{\varphi x}\xrightarrow{\psi_x}P_x$. Here $\eta_x=f_{x\varphi(x)}$, and the map $\psi_x=rf_{\varphi(x)x}$ is chosen so that $(rf_{\varphi(x)x})(f_{x\varphi(x)}) = t f_{xx}$. One can also show that  $(f_{x\varphi(x)})(rf_{\varphi(x)x}) = t f_{\varphi(x)\varphi(x)}$. 
We use this to construct new Frobenius and triangulated categories in the following way.

\begin{defn} Let $\varphi$ be and admissible automorphism of the cyclic poset $(X,c)_{\varphi}$. Define the category $\mathcal F_{\varphi}(X)$ to be the full subcategory of the Frobenius category $\mathcal F(X)$ consisting of all $(P,d)$ where $d: P\to P$ factors through $\eta_P: P\to \varphi P$.
\end{defn}

\begin{thm}[Theorem 1.4.7. \cite{IT15a}] Let $\varphi$ be an admissible automorphism of the cyclic poset $(X,c)_{\varphi}$. Then the category $\mathcal F_{\varphi}(X)$ is Frobenius category with the indecomposable projective-injective objects being $E(x,\varphi(x))=\left(P_x\oplus P_{\varphi (x)}, \ \varphi=
\begin {bmatrix}0&\psi_x \\ \eta_x&0\end{bmatrix}\right).$
\end{thm}

There are some known and some new  categories with cluster structures, which can be obtained as stable categories of such twisted Frobenius categories: cluster categories of type $A_n$, infinity-gon, $m$-cluster category of type $A_{\infty}$, etc. (many of these can be found in 2.6 Chart of Examples \cite{IT15a}). We will give precise descriptions in   Sections \ref{phi=id}, \ref {phi=canonical} and \ref{phi=theta}.

\begin{defn} \label{twisted stable category} Let $\varphi$ be an admissible automorphism of the cyclic poset $(X,c)$. Let $\mathcal F_{\varphi}(X)$ be the Frobenius category. Define the twisted stable category $\mathcal C_{\varphi}(X):= \underline{\mathcal F_{\varphi}}(X)$\end{defn}
In order to describe compatible sets and cluster structures in $\mathcal C_{\varphi}(X)$, we will use the following proposition which relates Ext$_{\mathcal C_{\varphi}(X)}$ and Hom$_{\mathcal C_{\varphi}(X)}$.

\begin{prop} Let $\mathcal C_{\varphi}(X)$ be the twisted stable  category associated to the poset $(X,c)_{\varphi}$. Then  $\Sigma (E(x,y))= E(\varphi^{-1}y,\varphi^{-1}x)$ and this correspondence is functorial.
\end{prop}

\begin{eg}\label{eg: C(Zn)}
In the basic case of Example \ref{eg: Zn and phi}, the indecomposable objects are given, up to isomorphism, by the chords of a regular $n$-gon. The vertices correspond to elements of $Z_n$ and so the chord $\overline{xy}$ corresponds to the object $E(x,y)\cong E(y,x)$. Then $\varphi^{-1}$ is clockwise rotation by $2\pi/n$. The chord $\overline{\varphi^{-1}x\varphi^{-1}y}$ corresponding to $\Sigma E(x,y)=E(\varphi^{-1}x,\varphi^{-1}y)$ is given by rotating the chord $\overline{xy}$ clockwise by $\pi/n$. The $n$ sides of the regular $n$-gon correspond to $E(x,\varphi(x))$ and $E(x,\varphi^{-1}(x))$ which are projective-injective in $\mathcal F_\varphi(Z_n)$ and thus equal to zero in $\mathcal C_\varphi(Z_n)$. Only the remaining $\binom n2-n$ chords are nonzero. So, we need $n\ge4$.
\begin{center}
\begin{tikzpicture}[scale=.75]
\begin{scope}
\foreach \x in {10,30,50,70,90,110,130,150,170}
\draw[fill,very thick,rotate=\x] (3,-.5)--(3,.5) circle[radius=2pt];
\end{scope}
\begin{scope}[rotate=180]
\foreach \x in {10,30,50,70,90,110,130,150,170}
\draw[fill,very thick,rotate=\x] (3,-.5)--(3,.5) circle[radius=2pt];
\end{scope}
\foreach \y/\ytext in {30/x, 50/\varphi x, 10/\quad\varphi^{-1}x, 170/y, 190/\varphi y,150/\varphi^{-1}y\quad}
\draw[rotate=\y] (3.5,.5) node{$\ytext$};
\begin{scope}
\draw[thick] (2.3,1.9)--(-3,0);
\draw[thick,rotate=-20] (2.3,1.9)--(-3,0);
\draw (1,2) node{$E(x,y)$};
\draw (1.6,.5) node{$\Sigma E(x,y)$};
\draw (.95,-.25) node{$= E(\varphi^{-1}x,\varphi^{-1}y)$};
\end{scope}
\end{tikzpicture}
\end{center}

It is well-known that the indecomposable objects of the cluster category of type $A_{n}$ correspond to the chords of a regular $(n+3)$-gon \cite{CCS}. Thus $\mathcal C_\varphi(Z_{n+3})$ is equivalent to the cluster category of type $A_{n}$ for $n\ge4$.
\end{eg}

\subsection{Cluster structures} We now recall the notion of cluster structures as introduced in \cite{BIRSc} and \cite{BIRSm},  and describe several known and some new classes of the cyclic posets for which the stable categories $\underline{\mathcal FX}$ have cluster structures. First, we recall the definition of the quiver $Q_{\mathcal T}$ of a Krull-Schmidt category $\mathcal T$: the vertices of the quiver correspond to the indecomposable objects in $\mathcal T$ and the number of arrows $T_i\to T_j$ between two indecomposable objects $T_i$ and $T_j$ is given by the dimension of the space of irreducible maps $rad(T_i,T_j)/rad^2 (T_i,T_j)$. Here $rad(\  ,\ )$ denotes the radical in add$T$, where the objects are finite direct sums of objects in $\mathcal T$.

\begin{defn} \label{Cluster structures}  A \emph{cluster structure} on a Krull-Schmidt triangulated category $\mathcal C$ is a collection of sets $\mathcal T$, called \emph{clusters}, of mutually non-isomorphic indecomposable objects called \emph{variables} satisfying the following conditions:
\begin{enumerate}
\item For any cluster variable $T$ in any cluster $\mathcal T$ there is, up to isomorphism, a unique object $T^*$ not isomorphic to $T$ so that $\mathcal T^*:=\mathcal T\backslash T\cup T^*$ is a cluster;
\item There are distinguished triangles $T^*\rightarrow B\rightarrow T$ and  $T\rightarrow B'\rightarrow T^*$
so that $B$ is a minimal right add$(\mathcal T\backslash T)$-approximation of $T$ and $B'$ is a minimal
left add$(\mathcal T \backslash T)$-approximation of $T$. 
\item There are no loops or 2-cycles in the quiver $Q_{\mathcal T}$ of any cluster $\mathcal T$. 
\item The quiver $Q_{\mathcal T^*}$ of $\mathcal T^*$ is obtained from the quiver $Q_{\mathcal T}$ of $\mathcal T$ by a Derksen-Weyman-Zelevinsky mutation;
\item If $\mathcal T'$ is obtained from $\mathcal T$ by replacing each variable with an isomorphic object,
then $\mathcal T'$ is a cluster.
\end{enumerate}
\end{defn}

The continuous cluster category is an example of a triangulated Krull Schmidt category with cluster structure as in the Definition \ref{Cluster structures}: actually this category was the original motivation for introducing cyclic posets.

\begin{eg} Let $(X,c)$ be the cyclic poset from Example \ref{cocycle10}, i.e. $X=S^1$ and the cocycle $c(x,y,z)= 1$ if ordered $x,y,z\in S^1$ complete the circle and $c(x,y,z)= 0$ if ordered $x,y,z\in S^1$ do not complete the circle.
\begin{itemize}
\item Indecomposable objects $E(x,y)$ in the Frobenius category $\cF X$ correspond to the geodesics between the points $x,y\in S^1$.
\item The objects $E(x,x)$ are projective-injective and correspond to the points on $S^1$.
\item The stable category $\underline{\mathcal FX}$ is the continuous cluster category $\mathcal C_{cont}$ of \cite{IT15b}.
\end{itemize}
\end{eg}

\section{Cluster structures on cyclic subposets of $S^1$ with automorphism $\varphi = id$}\label{phi=id}\label{subsets}
In this section we describe two cluster structures on 
the cyclic posets $(X,c)$ with $X\subset S^1$
which are obtained by using the admissible automorphism $\varphi=id$.  When $X=S^1$ we obtain the continuous cluster category $\mathcal C_{cont}=\cC_{id}(S^1)$  (see Section \ref{CCC}, this category was studied and described in two different ways in \cite{IT15a} and \cite{IT15b}).
 Another family  $\mathcal C_n=\cC_{id}(Z_n)$ (see Section \ref{SOC}), are the `spaced out' cluster categories which are obtained by taking $X=Z_n=n$-gon, however these categories are different from the classical cluster categories  given by triangulations of $n$-gon, associated to the Dynkin quiver $A_{n-3}$. A relation among these is given in Proposition \ref{spaced-out embedding}.

In the first case we summarize results from \cite{IT15a} and in the second we give description of compatibility, theorem about cluster structure, geometric interpretation and two figures describing the clusters.

\subsection{Compatibility and Clusters} 
We point out that the notion of compatibility used to define cluster structures on a triangulated categories can vary. Here we give the definition for the categories with automorphism $\varphi=id$ which is different from the condition in the classical  cluster categories.

\begin{defn} \label{Ext} Let $\mathcal C_{id}(X)$
be the triangulated category
associated to a cyclic poset $(X,c)_{id}$. Two indecomposable objects $E(x,y)$ and $E(x',y')$ in $\mathcal C_{id}(X)$ are said to be \emph{compatible} if 
$$\dim_k\Ext^1_{\mathcal C_{id}(X)}(E(x,y),E(x',y'))+\dim_k\Ext^1_{\mathcal C_{id}(X)}(E(x',y'),E(x,y))\leq 1.$$
\end{defn}

\begin{lem} \label{Hom} Let  $\mathcal C_{id}(X)$ be the triangulated category associated to a cyclic poset $(X,c)_{id}$. Then $E(x,y)$ and $E(x',y')$ are  \emph{compatible} if 
$$\dim_k\Hom_{\mathcal C_{id}(X)}(E(x,y),E(x',y'))+\dim_k\Hom_{\mathcal C_{id}(X)}(E(x',y'),E(x,y))\leq 1.$$
\end{lem}
\begin{proof} This follows immediately from Definition \ref{Ext} since Hom$=$Ext$^1$ when $\varphi$ is the identity automorphism, i.e. in the category $\mathcal C_{id}(X)$.
\end{proof}

\subsection{Continuous Cluster Categories}\label {CCC} The category $\mathcal C_{cont}$ is of the form $\mathcal C_{id}(S^1)$. It is a triangulated category with the cluster structure where each cluster $\mathcal T$ is a maximal collection of mutually compatible objects, which satisfy Hom condition of  Lemma \ref{Hom}. 

\begin{rem} 
While the continuous cluster category has a cluster structure, this category is quite different from the classical cluster categories of \cite{BMRRT}, \cite {CCS} or \cite{Am}. Here are some of its interesting properties: 
\begin{enumerate}
\item All clusters $\mathcal T$ are isomorphic to each other  (Theorem 5.2.1 in \cite{IT15b}). Thus, there is essentially only one cluster in $\cC_{cont}$. Again, in most known cluster categories, clusters are not isomorphic to each other.
\item All clusters, i.e. subcategories $\mathcal T$ are infinite since one of them is infinite (Proposition 4.2.7 in \cite{IT15b}). This is quite different from the situations when clusters correspond to modules or objects in the cluster categories.
\item The clusters $\mathcal T$ are not functorially finite subcategories. In other situations, when clusters are objects, they are always functorially finite. 

In order to see that the clusters in $\mathcal C_{cont}$ are not functorially finite (contravariantly finite and covariantly finite), it is enough to find one cluster $\mathcal T$ (since all of them are isomorphic) and show that it is not contravariantly finite. Hence, for a fixed cluster $\mathcal T$ it is enough to find  at least one object $M$ in $\mathcal C_{cont}$ such that there is no contravariant  $\add \mathcal T$-approximation of $M$. Recall that contravariant $\mathcal T$-approximation of $M$ is a morphism $f_M:T_M\to M$, with $T_M$ in $add \mathcal T$,  such that any map $\alpha :T'\to M$, with $T'$ in $add \mathcal T$  factors though $f_M$. 

For this, it is convenient to use the original description of $\mathcal C_{cont}$ from \cite{IT15b} as the orbit category of the triangulated category $\mathcal D$:  the indecomposable objects in $\mathcal D$ are $\{M(x,y)\in \mathbb R^2\ |\  |x-y|< 1\}\subset \mathbb R^2$  and Hom$_{\mathcal D}(M(x,y),M(x',y'))= k$ if the slope $(y'-y)/(x'-x)\geq 0$ and otherwise is $0$. The functor $F:\mathcal D\to \mathcal D$ defined as $F(M(x,y)):= M(y+1,x+1)$ is used to define orbit category $\mathcal C :=\mathcal D/F$ and it is proved in \cite{IT15b} that  this category has cluster structure and $\mathcal C$ was called continuous cluster category and denoted by $\mathcal C_{cont}$. It was shown in \cite{IT15a} that this category is isomorphic to the cluster category $\mathcal C_{id}(S^1)$ as defined in the Def.\ref{twisted stable category}.

A collection of objects (orbits of points $M(x,y)$ in $\mathcal D$) given by: \\
$M^{(n)}_j=((1-j)/2^n, (2^n-j)/2^n))$, for integers $n\geq 0$ and $0\leq j < 2^{n}$ forms a cluster in $\mathcal C_{cont}$  (Proposition 4.2.7. in \cite{IT15b}, with small notational adjustment).
With that, one can see that $M(1/3,1/3)$ does not have $add \mathcal T$-approximation (actually this is true for any $M(x,y)$ where either $x\neq m/2^k$ or $y\neq n/2^k$ for any $k\in \mathbb Z_{\geq 0}$).

\item Objects in $\mathcal T$ are not mutually Ext$^1$-orthogonal. Compatibility is given instead by the Ext-condition of Definition \ref{Ext} or, equivalently, by the Hom-condition of Lemma \ref{Hom}. 
\item {Unlike the other known cases of cluster categories, the category $\mathcal C\!_{cont}$ is not  $m$-Calabi-Yao for any $m$ since $X[m]\cong X$ for every object $X$. The definition of $m$-CY is that $D\Hom(X,Y)\cong\Hom(Y,X[m])$. However, there are compatible $X,Y$ where $\Hom(X,Y)=0$ but $\Hom(Y,X[m])=\Hom(Y,X)\neq0$. So, $\mathcal C\!_{cont}$ is not  $m$-CY.}
\end{enumerate}
\end{rem}

\begin{rem} Correspondence with hyperbolic plane. Indecomposable objects $E(x,y)$ of the continuous cluster category $\mathcal C\!_{cont}$ correspond to geodesics in the hyperbolic plane.
Clusters in $\mathcal C\!_{cont}$  correspond to ideal triangulations in the hyperbolic plane.

More precisely:
The hyperbolic plane $\mathfrak h^2$ is embedded conformally onto the open unit disk in $\RR^2$ which means that angles are preserved or, equivalently, the metric is dilated in the same proportion in every direction. The hyperbolic metric is given by
\[
	d\lambda=\frac{dr}{1-r^2}
\]
where $r$ is the Euclidean distance to the origin. This is equivalent to saying that the geodesics are the straight lines through the origin and circles (the portion inside $\mathfrak h^2$) which meet the unit circle centered at the origin at two right angles. These circles are always centered outside the unit disk.

The unit circle is sometimes called the \emph{ideal boundary} of $\mathfrak h^2$ since the points on the unit circle are, in hyperbolic metric, infinitely far away from the origin and thus not actually elements of $\mathfrak h^2$. In fact, the hyperbolic distance from a point to the origin is
\[
	\lambda=\int_0^r \frac{dt}{1-t^2}=\frac12 ln\left|
	\frac{1+r}{1-r}
	\right|
\]
which goes to $\infty$ as $r\to 1$. The indecomposable object $E(x,y)$ corresponds to the unique geodesic in $\mathfrak h^2$ with boundary (limit points) $x,y$. The geodesics corresponding to $E(x,y)$ and $E(x',y')$ will \emph{cross} (meet in the interior $\mathfrak h^2$ of the unit disk) if and only if they are `crossing', or equivalently, are not compatible, by Remark \ref{compatible/noncrossing} For example, in Figure \ref{Fig: hyperbolic plane}, $E(x,z)$ is compatible with $E(x,y)$ since morphisms are given by counter clockwise rotation and any rotation of $\widehat{xz}$ to $\widehat{xy}$ factors though $\widehat{xx}$ which corresponds to $E(x,x)=0$ and we get $\Hom(E(x,z),E(x,y))=0$ in $\cC_{id}(S^1)$.
\end{rem}

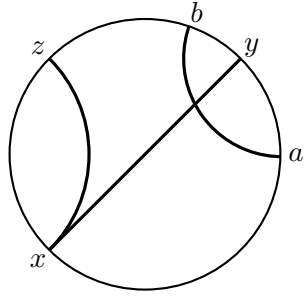
\begin{figure}[htbp]
\begin{center}
\begin{tikzpicture}[scale=.6]
\draw[thick] (0,0) circle [radius=3cm];
\draw[rotate=225] (3.35,0) node{$x$};
\draw[rotate=45] (3.35,0) node{$y$};
\draw[rotate=135] (3.35,0) node{$z$};
\draw[rotate=0] (3.35,0) node{$a$};
\draw[rotate=70] (3.35,0) node{$b$};
\clip (0,0) circle [radius=3cm];
\draw[very thick,rotate=45] (-3,0)--(3,0);
\draw[very thick, rotate=35] (3.708,0) circle [radius=2.18cm];
\draw[very thick, rotate=180] (4.24,0) circle [radius=3cm];
\end{tikzpicture}
\caption{This is the `Poincar\'e disk' model of the hyperbolic plane $\mathfrak h^2$ where geodesics are either straight lines through the origin or semi-circles perpendicular to the boundary circle $S^1$ (which is strictly speaking disjoint from $\mathfrak h^2$). They are `noncossing' if they do not meet in their interiors. For example, the geodesics $\widehat{ab}$ and $\widehat{xy}$ cross but neither crosses $\widehat{xz}$}
\label{Fig: hyperbolic plane}
\end{center}
\end{figure}

\subsection{Spaced-out clusters}\label{SOC}

We will construct the `spaced-out cluster category'  $\cC_{id}(Z_n)$ as the cluster category of a finite subset $Z_n$ of $S^1$ and explain its relation to the standard cluster categories of type $A_m$ proved Proposition \ref{spaced-out embedding}) and illustrated in Figure \ref{fig-spaced}.

\begin{defn}For $n\ge3$ we define the \emph{spaced-out cluster category} $\cC_{id}(Z_n)$ to be the cluster category of the finite cyclically ordered set $Z_n$ with the \underline{trivial automorphism} $\varphi=id$. As in Example \ref{eg: Zn}, we take the elements of $Z_n\subset S^1$ to be $x_k=[2k\pi/n]$ for $k=1,2,\cdots,n$. ($x_{k+1}=x_k^+$ is the successor of $x_k$ and $x_1=x_n^+$ is the successor of $x_n$.)
\end{defn}

\begin{rem} Let $\cC_{id}(Z_n)$ be the spaced-out cluster category.
\begin{enumerate}
\item The category $\cC_{id}(Z_n)$ is a full subcategory of the continuous cluster category $\cC_{cont}=\cC_{id}(S^1)$. To see this, consider the Frobenius categories $\cF_{id}(Z_n)\subset\cF_{id}(S^1)$. Then the injective envelope in $\cF_{id}(S^1)$ of each object of $\cF_{id}(Z_n)$ lies in $\cF_{id}(Z_n)$ and a morphism in $\cF_{id}(Z_n)$ factors through a projective-injective in $\cF_{id}(S^1)$ if and only if it factors through an projective-injective object of $\cF_{id}(Z_n)$. Thus the inclusion functor of $\cC_{id}(Z_n)$ into $\cC_{id}(S^1)$ is full and faithful.
\item The indecomposable objects of  $\cC_{id}(Z_n)$ are $E(x,y)$ where $x\neq y\in Z_n$.
\item $\cC_{id}(Z_n)$ has $\binom n2=\frac{n(n-1)}2$ indecomposable objects $E(x,y)$ (up to isomorphism).
\end{enumerate}
\end{rem}

\begin{defn}
Compatibility condition 
is given by the Ext-condition of Definition \ref{Ext} or, equivalently, by the Hom-condition of Lemma \ref{Hom}, (the same as for $\cC_{cont}=\cC_{id}(S^1)$ since, in both cases, Hom=Ext).
A \emph{cluster} is defined to be a maximal collection of pairwise compatible  objects in $\cC_{can}(Z_n)$. 
\end{defn}

\begin{rem} Let $\cC_{id}(Z_n)$ be the spaced-out cluster category.
\begin{enumerate} \item A pair of objects $E(x,y),E(x',y')$ are compatible
if and only if they are noncrossing (See \cite[Prop 4.1.3]{IT15b}.) In particular, the $n$ objects $E(x,x^+)$ are compatible with every other object in $\cC_{can}(Z_n)$.
\item $E(x,y)[1]=E(\varphi^{-1}(y),\varphi^{-1}(x))=E(y,x)\cong E(x,y)$ for all $x\neq y$ in $Z_n$ since $\varphi=id$
\item  $E(x,y)=0$ if and only if $x=y$.
\end{enumerate} 
\end{rem}

\begin{prop}
The clusters of $\cC_{id}(Z_n)$ are in bijection with triangulations of the regular $n$-gon with vertices being the elements of $Z_n$, and where an edge $\overline{xy}$ is in the triangulation if and only if $E(x,y)$ is a member of the cluster.
\end{prop}

\begin{proof}
The edges $\overline{xy}$ corresponding to the objects $E(x,y)$ of a cluster form a maximal set of noncrossing edges in the regular polygon. These give the triangulations of the $n$-gon.
\end{proof}

\begin{cor}
The clusters of $\cC_{id}(Z_n)$ form a cluster structure where the objects $E(x,x^+)$ are frozen variables which belong to every cluster and every object $E(x,y)\in \cC_{id}(Z_n)$ belongs to at least one cluster.
\end{cor}

\begin{proof} Since $Z_n$ is finite, every object $E(x,y)$ in $\cC_{id}(Z_n)$ can be completed to a maximal compatible set which is a cluster by definition.

Given any nonfrozen object $E(x,y)$ in a cluster $\cT$, the corresponding edge $\overline{xy}$ belongs to exactly two triangles whose union is a quadrilateral. The quadrilateral has two diagonals: one is $\overline{xy}$ and the other $\overline{ab}$ corresponds to the unique object $E(a,b)$ which can replace $E(x,y)$ in the cluster $\cT$. After possibly switching $a,b$ the points $x,a,y,b$ will be in cyclic order around the circle $S^1$ and we have an exact sequence in the Frobenius category:
\[
	0\to E(x,y)\to E(x,b)\oplus E(a,y)\to E(a,b)\to 0
\]
which gives a distinguished triangle in $\cC_{id}(Z_n)$ making the middle term $E(x,b)\oplus E(a,y)$ a left add$(\cT\backslash E(x,y))$-approximation of $E(x,y)$ as required for the condition (2) for cluster structures in  Definition \ref{Cluster structures}.
\end{proof}

The name `spaced-out cluster category' comes from the fact that $\cC_{id}(Z_n)$ is embedded in the standard cluster category of $A_{2n-3}$ as a maximal collection of objects in the Auslander-Reiten quiver which are not connected by irreducible maps. (Figure \ref{fig-spaced} and Proposition \ref{spaced-out embedding}.)

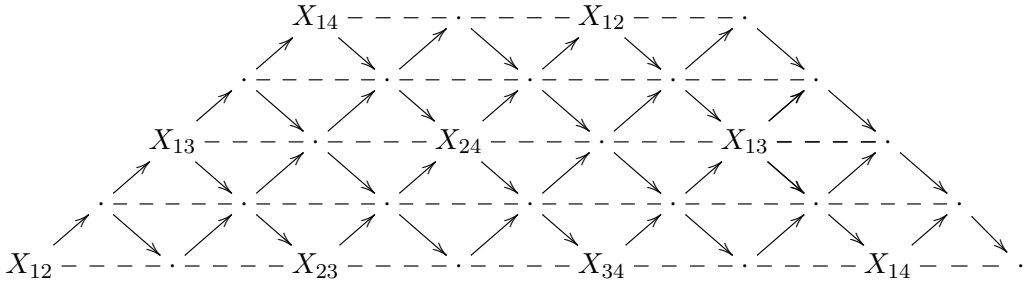
\begin{figure}[htbp]
\begin{center}
\[
\xymatrixrowsep{10pt}\xymatrixcolsep{10pt}
\xymatrix{
  & & & &X_{14}\ar[dr]\ar@{--}[rr]& &\cdot\ar[dr]\ar@{--}[rr]&&X_{12}\ar[dr]\ar@{--}[rr]&&\cdot\ar[dr]\\ 
  & & &\cdot\ar[ur]\ar@{--}[rr]\ar[dr]& &\cdot\ar[ur]\ar@{--}[rr]\ar[dr]&&\cdot\ar[ur]\ar@{--}[rr]\ar[dr]&&\cdot\ar[ur]\ar@{--}[rr]\ar[dr]&&\cdot\ar[dr]\\ 
 & & X_{13}\ar[ur]\ar@{--}[rr]\ar[dr]& &\cdot\ar[ur]\ar@{--}[rr]\ar[dr]&&X_{24}\ar[ur]\ar@{--}[rr]\ar[dr]&&\cdot\ar[ur]\ar@{--}[rr]\ar[dr]&&X_{13}\ar[ur]\ar@{--}[rr]\ar[dr]\ar[ur]\ar@{--}[rr]\ar[dr]&&\cdot\ar[dr]\\ 
 & \cdot\ar[ur]\ar@{--}[rr]\ar[dr] & &\cdot\ar[ur]\ar@{--}[rr]\ar[dr]&&\cdot\ar[ur]\ar@{--}[rr]\ar[dr]&&\cdot\ar[ur]\ar@{--}[rr]\ar[dr]&&\cdot\ar[ur]\ar@{--}[rr]\ar[dr]&&\cdot\ar[ur]\ar@{--}[rr]\ar[dr]&&\cdot\ar[dr]\\ 
 X_{12}\ar[ur]\ar@{--}[rr]& &\cdot\ar[ur]\ar@{--}[rr]&&X_{23}\ar[ur]\ar@{--}[rr]&&\cdot\ar[ur]\ar@{--}[rr]&&X_{34}\ar[ur]\ar@{--}[rr]&&\cdot\ar[ur]\ar@{--}[rr]&&X_{14}\ar[ur]\ar@{--}[rr]&&\ \cdot\ \\ 
	}
\]
\caption{The embedding $J:\cC_{id}(Z_4)\into\cC_\varphi(Z_{8})$ of the spaced-out cluster category of $Z_4$ into the cluster category of type $A_5$. Here $E(x_{2i},x_{2j})=JE(x_i,x_j)$ is denoted $X_{ij}$. $J$ does not commute with shift $[1]$ since $J\cC_{id}(Z_4)\cap J\cC_{id}(Z_4)[1]=0$}
\label{fig-spaced}
\end{center}
\end{figure}

\begin{prop}\label{spaced-out embedding}
There is an embedding $J:\cC_{id}(Z_n)\to \cC_\varphi(Z_{2n})$ given on objects by 
\[
	JE(x_i,x_j)=E(x_{2i},x_{2j}).
\]
On morphism, $J$ is linear and takes basic morphisms to basic morphisms. Furthermore, $J$ sends clusters to clusters.
\end{prop}

\begin{proof}
We first construct an embedding on the level of the Frobenius categories $\cF_n:=\cF_\varphi(Z_n)$ and $\cF_{2n}:=\cF_\varphi(Z_{2n})$. The first category $\cF_n$ is a full subcategory of the Frobenius category $\cF(Z_n)$ whose stable category is $\cC_{id}(Z_n)$. The only objects in $\cF(Z_n)$ which are not in $\cF_n$ are the projective-injective objects $E(x,x)$. Therefore, the composition $\cF_n\into \cF(Z_n)\onto \cC_{id}(Z_n)$ is surjective.
Consider the following diagram where $\widetilde J:\cF_n\into \cF_{2n}$ is the inclusion functor of the full subcategory $\cF_\varphi(Z_n)$ of $\cF_\varphi(Z_{2n})$ consisting of the objects $E(x_{2i},x_{2j})$ where $i\neq j$.
\[
\xymatrix{
\cF_n=\cF_\varphi(Z_n)\ar[d]^{\widetilde J}\ar[r] & \cF(Z_n)\ar[r] &
	\cC_{id}(Z_n)\ar[d]^J\\
\cF_{2n}=\cF_\varphi(Z_{2n}) \ar[rr]& &
	\cC_\varphi(Z_{2n})
	}
\]
To prove the Proposition is suffices to prove two things:
\begin{enumerate}
\item The diagram commutes.
\item A morphism $f$ in $\cF_n$ goes to zero in $\cC_{id}(Z_n)$ ($\underline f=0$) if and only if the corresponding morphism $\widetilde J f$ in $\cF_{2n}$ goes to zero in $\cC_\varphi(Z_{2n})$, i.e., $\underline f=0\Leftrightarrow\underline {\widetilde Jf}=0$.
\end{enumerate}
In fact, it suffices to prove (2) since, in that case, there is a unique induced functor $J:\cC_{id}(Z_n)\to \cC_\varphi(Z_{2n})$ making the diagram commute. This statement follows from the fact that the functor $\cF_n\to \cC_{id}(Z_n)$ is a bijection on objects.

We prove (2) directly from the definition of the stable category. A morphism $f:E(x,y)\to E(x',y')$ in $\cF_n$ goes to zero in $\cC_{id}(Z_n)$ if and only if the image of $f$ in $\cF(Z_n)$ factors through the injective envelope $E(x,x)\oplus E(y,y)$ of $E(x,y)$ in $\cF(Z_n)$. But this is equivalent to $\widetilde f$ factoring through $E(x,x^-)\oplus E(y^-,y)$ (this is equivalent to the fact that, for even integers $x,y$, $x<y$ iff $x<y-1$) which is equivalent to $\widetilde J f$ mapping to zero in $\cC_\varphi(Z_{2n})$.

 Finally, any basic morphism in $\cC_{id}(Z_n)$ comes from a basic morphism in $\cF_n$ which maps to a basic morphism in $\cF_{2n}$ which maps to a basic morphism in $\cC_\varphi(Z_{2n})$. So, $J$ takes basic morphisms to basic morphisms.
\end{proof}

\begin{figure}[htbp]
\begin{center}
\begin{tikzpicture}[scale=.8]
\draw[thick] (0,0) circle [radius=3cm];
\draw[very thick,rotate=36] (2.85,0)--(3.15,0);
\draw[very thick,rotate=108] (2.85,0)--(3.15,0);
\draw[very thick,rotate=180] (2.85,0)--(3.15,0);
\draw[very thick,rotate=252] (2.85,0)--(3.15,0);
\draw[very thick,rotate=324] (2.85,0)--(3.15,0);
\clip (0,0) circle [radius=3cm];
\draw[very thick, rotate=36] (3.708,0) circle [radius=2.18cm];
\draw[very thick, rotate=108] (3.708,0) circle [radius=2.18cm];
\draw[very thick, rotate=180] (3.708,0) circle [radius=2.18cm];
\draw[very thick, rotate=252] (3.708,0) circle [radius=2.18cm];
\draw[very thick, rotate=324] (3.708,0) circle [radius=2.18cm];
\draw[very thick, rotate=72] (9.7,0) circle [radius=9.23cm];
\draw[very thick, rotate=-72] (9.7,0) circle [radius=9.23cm];
\end{tikzpicture}
\caption{A cluster $\cT$ in $\cC_{id}(Z_5)$ has 7 objects whose corresponding geodesics give an ideal triangulation of an ideal pentagon. Adding 5 points on the boundary we get a triangulation of the 10-gon representing a cluster $J\cT$ in $\cC_\varphi(Z_{10})$.}
\label{Figure99}
\end{center}
\end{figure}
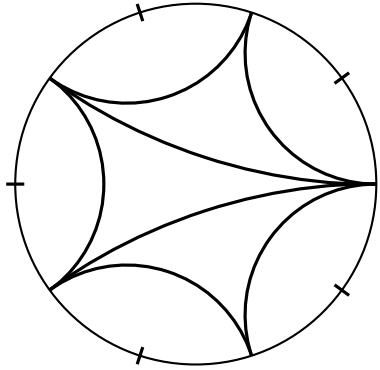

\section{Cluster structure on the cyclic subposets of $S^1$ with canonical automorphism $\varphi$}  \label{phi=canonical}

In this section we consider  admissible discrete subsets $Z$ of $S^1$ and canonical automorphisms $\varphi$ which are used to define triangulated categories $\cC_{\varphi}(Z)$ with cluster structure (see Theorem \ref{Ccanonical}). We address the question of which clusters in these cluster categories give a topological triangulation of the 2-disk or part of the 2-disk. The first main result (Lemma \ref{1-1}) is that a cluster gives a topological triangulation of an open subset of $D^2$ if an only if it is locally finite. In section \ref{CcanEx} we construct several categories with cluster structures depending on the subsets $Z\subset S^1$ {\color{black}and we give a complete description of the locally finite clusters in these examples.} In Section \ref{triangulation} we describe `triangulation clusters', i.e. clusters which define triangulation of the unit disk $D^2=\{(x,y)\in\RR^2\,:\,x^2+y^2\le1\}$ after the limit points $L(Z)$ of $Z$ are removed (Theorem \ref{triangulationtheorem}).
In Section \ref{non-triangulation} we consider {\color{black}the locally finite} clusters which {\color{black}are not triangulation clusters, i.e.,} do not define triangulation on $D^2\backslash L(Z)$ and we associate {\color{black}new kinds of} cyclic posets and spaces, called {\color{black}`cactus cyclic posets' and} `cactus spaces' in order for cluster to become triangulation cluster (Figure \ref{fig04}).

\subsection{Admissible discrete subsets of the circle $S^1$ and canonical automorphism}\label{CcanEx}
Let $Z$ be a subset of $S^1$, let $\pi: \mathbb R\to S^1$ be the covering and let $\tilde Z:= \pi^{-1}Z$. Let $p:\tilde Z\to Z$ be the restricted projection map.

\begin{defn}\label{def: admissible subsets of S1}
A subset $Z$ of $S^1$ will be called \emph{admissible discrete subset} if it satisfies:
\begin{enumerate}
\item $Z$ is \emph{discrete} in the sense that every element of $Z$ has an open neighborhood $U\subseteq S^1$ which contains no other element of $Z$.
\item $Z$ satisfies the following \emph{two-sided limit condition}. Every point $x$ in $S^1$ which is a limit point of $Z$ is a limit from both sides, i.e., any lifting $\tilde x\in\RR$ of $x$ is both an ascending and descending limit of $\tilde Z$.
\item $Z$ has at least four elements.
\end{enumerate}
\end{defn}

In order to define canonical admissible automorphism, we need to
 define successors and predecessors of elements in a cyclic poset. Let $x\in Z\subset S^1$. Choose a lifting $\tilde x$ of $x$ to $\RR$ and let $\tilde x^+$ be the infimum of all elements of $\tilde Z$ which are greater than $\tilde x$. Since $\tilde x^+$ is not an ascending limit of elements of $\tilde Z$, it cannot be a descending limit by the two-sided limit condition. Therefore, $\tilde x^+\in\tilde Z$. Similarly let $\tilde x^-\in\tilde Z$ be the supremum of all elements of $\tilde Z$ less than $\tilde x$. 
If $x\in Z$ then $\tilde x$ is the only element of $\tilde Z$ which lies between $\tilde x^+$ and $\tilde x^-$.

\begin{defn} Let $z$ be an element of an admissible subset $Z$ of $S^1$. We define the \emph{successor} $z^+$ and \emph{predecessor} $z^-$ of  $z\in Z$ to be the images in $S^1$ of $\tilde z^+,\tilde z^-$ for any lifting $\tilde z$ of $z$ in $\RR$. The successor and predecessor are well defined, independent of the lifting.
\end{defn}

\begin{defn}\label{def: canonical admissible automorphism}
Given an admissible subset $Z$ of the circle $S^1$, the \emph{canonical admissible automorphism} of $Z$ is given by $\varphi(z)= z^+$ and $\tilde\varphi(\tilde z)=\tilde z^+$. 
\end{defn}

The following theorem, which was proved in \cite{IT15b}, supplies collections of many categories which have  cluster structures:

\begin{thm}[Lemma 2.4.4, \cite{IT15a}]\label{Ccanonical}Let $\mathcal C_{can}(X)$ be the stable category associated to a cyclic poset $(X,c)_{\varphi}$, where $X$ is an admissible discrete subset of $S^1$ and $\varphi$ is the canonical admissible automorphism. Then the triangulated category $\mathcal C_{can}(X)$ has a cluster structure.
\end{thm}

We now give an example of a discrete subset of $S^1$ which is not admissible. 

\begin{eg} The following is an example of a discrete subset of $S^1$ which is not admissible since it does not satisfy  the limit condition: 
\[Z=\left\{\frac{\pi}n\,:\, n\text{ is a positive integer }\right\}\subset S^1\]
The limit point $0$ is a descending limit but not an ascending limit of $\tilde Z$. Every element $\frac{\pi}n$ has a predecessor $\frac{\pi}{n+1}$ and every element $\frac{\pi}{n}$ for $n\ge2$ has a successor $\frac{\pi}{n-1}$ except for the point ${\pi}\in Z$ which has no successor since the infimum of the set of all $x>\pi\in \tilde Z$ is $2\pi$ which is not contained in $\tilde Z$. Condition (2) is violated because $2\pi$ is a descending limit of $\tilde Z$ but not an ascending limit.
\end{eg}

We now recall a general construction of a new cyclic poset $\mathcal P*\mathcal O$ which is created out of a pair of a cyclic poset $\mathcal P$ and ordered set $O$ (Definition 1.1.14 \cite{IT15b}). 

\begin{defn} Let $\mathcal P$ be a cyclic poset and let $O$ be an ordered set. We give the cyclic poset structure on $\cP\ast \mathcal O$ by defining its covering poset to be the Cartesian product $\widetilde \cP\times \mathcal O$ with lexicographic order. Let $\sigma(x,a):=(\sigma(x),a)$ where $\sigma:\widetilde{\mathcal P}\to \widetilde{\mathcal P}$ is the defining automorphism of $\widetilde{\cP}$ Def.\ref{cover}(2). This is a poset automorphism and, for any pair of elements $(x,a),(y,b)$ in $\widetilde \cP\times \mathcal O$ there is an integer $n$ so that $x< \sigma(y)$ which makes $(x,a)<\sigma(y,b)$. So, we get a cyclic poset.

\end{defn}

Of particular interest for this paper will be the case when $\mathcal P$ is a cyclic poset obtained from a discrete admissible subset $Z$ of $S^1$ and $O$ is ordered set $Z_{\infty}$ which is isomorphic to $\mathbb Z$. The following example is a convenient description of such a cyclic poset together with an embedding into $S^1$.

\begin{eg}\label{examples of Z}
Let $Z\subset S^1$ be any discrete admissible set satisfying the conditions of the Definition \ref{def: admissible subsets of S1}. Then we can form another set 
\[Z(Z_{\infty}):=\bigcup_{z\in Z} \left\{\frac{z^++z}2+\frac{z^+-z}\pi \arctan n\,:\,n\in\ZZ\right\}\]
which is given by inserting a copy of $Z_\infty$ between any two consecutive elements of $Z$, then deleting the original set $Z$. Then $Z_\infty(Z)$ also satisfying the definition, and the closure of $Z$ in $S^1$ is the set of limit points of the set $Z_\infty(Z)$.
\end{eg}

\begin{rem} \label{L(Zinfty)} Let $Z\subset S^1$ be a discrete admissible subset.
\begin{enumerate}
\item We use the notation $Z(Z_{\infty})$ for $Z*Z_{\infty}$ since this construction in this case inserts a copy of $Z_{\infty}$ between any two adjacent points in $Z$, i.e. between $z$ and $z^+$, in such a way that the points of $Z$ all become limit points. 
\item Let $Z\subset S^1$ be a discrete admissible subset. Suppose that $L$, the set of limit points of $Z$ is finite. Then $Z=L(Z_{\infty})$. 
\end{enumerate}
\end{rem}

\begin{eg}\label{examples of Z2(Zinfty)}
As a special case of Example \ref{examples of Z}, let $Z_2=\{0,\pi\}$. Then \[Z_2(Z_{\infty})=\left\{\arctan n+\frac\pi2\,:\, n\in\ZZ\right\}\cup \left\{\arctan n-\frac\pi2\,:\, n\in\ZZ\right\}\]which has two limit points $\{0,\pi\}$. Let $\mathcal C_{can}(Z_2(Z_{\infty}))$ be the associated category with cluster structure as in Theorem \ref{Ccanonical}. This example is very similar to the one in \cite{LP}.
\end{eg}

We will label the points on $S^1$ in Figure \ref{FigureZ2} by: \\
$a^+_i \rightsquigarrow (\arctan i+\frac\pi2)$ and $a^-_i \rightsquigarrow (\arctan i-\frac\pi2)$. \\

We will label the objects in $\mathcal C_{can}(Z_2(Z_{\infty}))$ as follows:\\
 $E_{ij}=E(a^+_i,a^+_j)$ (these objects correspond to the arcs in the upper half of $D^2$),\\
$F_{ij}=E(a^-_i,a^-_j)$ (these correspond to the arcs in the lower half of $D^2$),\\
$G_{ij}=E(a^+_i,a^-_j)$ (correspond to the arcs between the upper half and lower half of $D^2$).\\
With this notation, we will label in Figure \ref{FigureZ2} certain objects and their corresponding arcs all of which are compatible to each other and hence form a subset of a cluster: \\$E_{-20}, E_{02}, E_{03}, E_{35}, F_{-20}, F_{02}, F_{04}, F_{24}, G_{3,-2}, G_{30}, G_{00}, G_{04}, G_{24}$. \\

\begin{figure}[htbp]
\begin{center}
\begin{tikzpicture}[scale=.8]
\draw[thick] (0,0) circle [radius=3cm];
\draw[thick] (3,0) node{\huge$\ast$};
\draw[thick] (-3,0) node{\huge$\ast$};
\foreach \x in {90,45,14.04,6.34,3.58,2.29}\draw[thick,rotate=\x] (2.9,0)--(3.1,0) (-2.9,0)--(-3.1,0) ;
\foreach \x in {45,14.04,6.34,3.58,2.29}\draw[thick,rotate=-\x] (2.9,0)--(3.1,0)  (-2.9,0)--(-3.1,0) ;
\draw[rotate=90] (3.4,0) node{\tiny$a_0^+$};
\draw[rotate=90] (-3.4,0) node{\tiny$a_0^-$};
\draw[rotate=45] (3.4,0) node{\tiny$a_{-1}^+$};
\draw[rotate=14] (3.5,0) node{\tiny$a_{-2}^+$};
\draw[rotate=6.34] (3.5,0) node{\tiny$a_{-3}^+$};
\draw[rotate=-6.34] (3.4,0) node{\tiny$a_{3}^-$};
\draw[rotate=6.34] (-3.5,0) node{\tiny$a_{-3}^-$};
\draw[rotate=-6.34] (-3.4,0) node{\tiny$a_{3}^+$};
\draw[rotate=-45] (3.4,0) node{\tiny$a_1^-$};
\draw[rotate=-14] (3.4,0) node{\tiny$a_2^-$};
\draw[rotate=45] (-3.4,0) node{\tiny$a_{-1}^-$};
\draw[rotate=14] (-3.5,0) node{\tiny$a_{-2}^-$};
\draw[rotate=-45] (-3.4,0) node{\tiny$a_1^+$};
\draw[rotate=-14] (-3.4,0) node{\tiny$a_2^+$};
\draw[->] (-4.5,3) .. controls (-1.3,4) and (-1.3,2) ..(-.6,1.5);
\draw (-4.5,3) node[left]{\small$E_{0,2}=E(a_0^+,a_2^+)$};
\draw[->] (-4.5,2) .. controls (-2.3,2.5) and (-2,1) ..(-1.8,.5);
\draw (-4.5,2) node[left]{\small$E_{0,3}=E(a_0^+,a_3^+)$};
\draw[->] (-4.5,-2) .. controls (-2,-1.5) and (-1.8,-.4) ..(-2.5,-.2);
\draw (-4.5,-2) node[left]{\small$G_{0,-2}=E(a_0^+,a_{-2}^-)$};
\draw[->] (-4.5,-3) .. controls (-2.3,-2) and (-1.8,-.8) ..(-1.3,-.35);
\draw (-4.5,-3)node[left]{\small$G_{3,0}=E(a_3^+,a_0^-)$};
\draw[->] (4.5,-2) .. controls (2.3,-2.5) and (2,-1.5) ..(2,-.75);
\draw (4.5,-2)node[right]{\small$F_{0,2}=E(a_0^-,a_2^-)$};
\draw[->] (4.5,-1) .. controls (2.3,-2) and (2,-1) ..(2.7,-.5);
\draw (4.5,-1) node[right]{\small$F_{2,4}=E(a_2^-,a_4^-)$};
\clip (0,0) circle [radius=3cm];
\draw[very thick] (0,-3)--(0,3);
\lcluster{-22.5}{7.84}{7.24}; 
\lncluster{-22.5}{7.84}{7.24}; 
\lcluster{22.5}{7.84}{7.24}; 
\lncluster{22.5}{7.84}{7.24}; 
\lncluster{-37.98}{4.87}{3.84}; 
\cluster{5.23}{3.04}{.46}; 
\cluster{43.21}{4.38}{3.19}; 
\ncluster{41.83}{4.5}{3.35}; 
\ncluster{3.85}{3.05}{.54}; 
\rcluster{52.02}{3.81}{2.34}; 
\bcluster{-52.02}{3.81}{2.34}; 
\bncluster{52.02}{3.81}{2.34}; 
\rncluster{-52.02}{3.81}{2.34}; 
\rncluster{-48.17}{4.03}{2.69}; 
\bcluster{-8.81}{3.01}{.27}; 
\bcluster{-46.79}{4.12}{2.82}; 
\end{tikzpicture}
\caption{Some compatible objects in $\mathcal C_{can}(Z_2(Z_{\infty}))$}
\label{FigureZ2}
\end{center}
\end{figure}
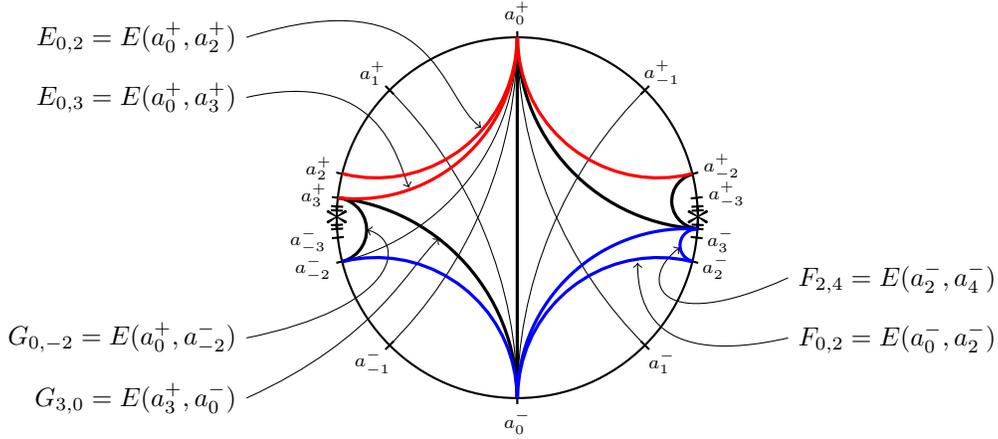

The Auslander-Reiten quiver of $\mathcal C_{can}(Z_2(Z_{\infty})$ is shown in Figure \ref{AR}. The quiver has three components, two of type $\ZZ A_\infty$: one containing all $\{E_{ij}\}$ objects, one containing all $\{F_{ij}\}$ objects and one of type $\ZZ A_\infty^\infty$ which contains all $\{G_{ij}\}$ objects. Of special importance will be collections of objects which appear on `complete zig-zag's which we now define.

\begin{defn} A sequence of indecomposable objects $\{ \dots , X_{-1}, X_0, X_1, \dots ,X_i, \dots\}$ forms a \emph{zig-zag} sequence if each pair $X_i, X_{i+1}$ is connected by an irreducible map in either direction. A \emph{complete zig-zag} sequence is a maximal such sequence.
\end{defn}
\begin{rem} \label{rem} Let $Z=Z_2(Z_{\infty})$ and let $\mathcal C_{can}(Z)$  be the triangulated category of the Theorem \ref{Ccanonical}. Here we describe some special clusters which will be exactly the clusters that define topological triangulations of the  $D^2\backslash L(Z)$, where $L(Z)$ is the set of limit points of $Z$ (in this example we have $L(Z)=\{0,\pi\}$) ('triangulation clusters' in Section \ref{triangulation}).
\begin{enumerate}
\item Particularly nice clusters are `complete zig-zag's in the component of $\{G_{ij}\}$. 
\item If $\mathcal T$ is a `complete zig-zag' cluster in  the component $\{G_{ij}\}$ and if  $\mathcal T$ is mutated at a corner object of the `zig-zag', then the new cluster is still a complete `zig-zag' cluster in  the component $\{G_{ij}\}$. 
\item If $\mathcal T$ is a `complete zig-zag' cluster in  the component $\{G_{ij}\}$ and if  $\mathcal T$ is mutated at an object on the positive slope in the AR-quiver, then the new object  is  in  the component $\{E_{ij}\}$. 
\item If $\mathcal T$ is a `complete zig-zag' cluster in  the component $\{G_{ij}\}$ and if  $\mathcal T$ is mutated at an object on the negative slope in the AR-quiver, then the new object  is  in  the component $\{F_{ij}\}$. 

\end{enumerate}

\end{rem}

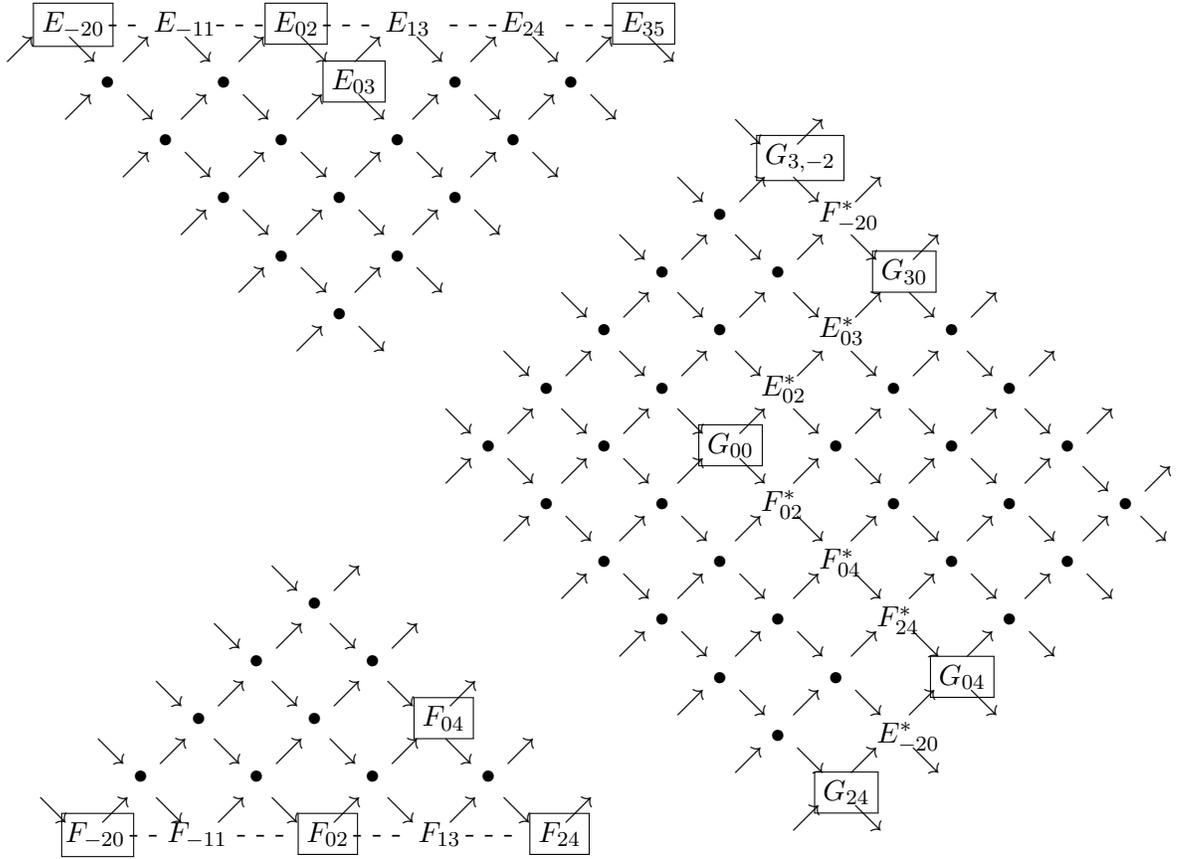
\begin{figure}[htbp]
\begin{center}
%
{
\setlength{\unitlength}{2.2cm}
{\mbox{
\begin{picture}(6,5)
      \thicklines
  \put(0.3,.3){ 
  %
  %
  \put(-.5,-0.1){
\put(0,4.7){
\put(0,0){\arUD{\!\boxed{E_{-20}}\!\text{- - }}}
\put(0.7,0){\arUD{{E_{-11}}\text{- - -}}}
\put(1.4,0){\arUD{\!\boxed{E_{02}}\!\text{- - -}}}
\put(2.1,0){\arUD{E_{13}\ \ \text{- - -}}}
\put(2.8,0){\arUD{E_{24}\ \ \text{- - -}}}
\put(3.5,0){\arUD{\!\boxed{E_{35}}}}
}
\put(.35,4.35){  \put(0,0){\arUD{\,\bullet}}
  \put(0.7,0){\arUD{\,\bullet}}
  \put(1.4,0){\arUD{\!\boxed{E_{03}}}}
  \put(2.1,0){\arUD{\,\bullet}}
  \put(2.8,0){\arUD{\,\bullet}}
  }
  \put(.7,4){
  \put(0,0){\arUD{\,\bullet}}
\put(0.7,0){\arUD{\,\bullet}}
\put(1.4,0){\arUD{\,\bullet}}
\put(2.1,0){\arUD{\,\bullet}}
  }
  \put(1.05,3.65){
  \put(0,0){\arUD{\,\bullet}}
\put(0.7,0){\arUD{\,\bullet}}
\put(1.4,0){\arUD{\,\bullet}}
  }
  \put(1.4,3.3){
  \put(0,0){\arUD{\,\bullet}}
\put(0.7,0){\arUD{\,\bullet}}
}
  \put(1.75,2.95){
  \put(0,0){\arUD{\,\bullet}}
}
}
%
%
%
  \put(2.15,0.35){
    \put(-.35,1.65){
  \put(4.2,0){\arDU{\,\bullet}}
  }
  \put(0,2){
  \put(0,0){\arDU{\,\bullet}}
  \put(0.7,0){\arDU{\,\bullet}}
  \put(1.4,0){\arDU{\!\!\boxed{G_{00}}}}
  \put(2.1,0){\arDU{\,\bullet}}
  \put(2.8,0){\arDU{\,\bullet}}
  \put(3.5,0){\arDU{\,\bullet}}
  }
  \put(0.35,2.35){
  \put(0,0){\arDU{\,\bullet}}
  \put(0.7,0){\arDU{\,\bullet}}
  \put(1.4,0){\arDU{\!E_{02}^\ast}}
  \put(2.1,0){\arDU{\,\bullet}}
  \put(2.8,0){\arDU{\,\bullet}}
  }
  \put(0.7,2.7){
  \put(0,0){\arDU{\,\bullet}}
  \put(0.7,0){\arDU{\,\bullet}}
  \put(1.4,0){\arDU{\!E_{03}^\ast}}
  \put(2.1,0){\arDU{\,\bullet}}
  }
  \put(1.05,3.05){
  \put(0,0){\arDU{\,\bullet}}
  \put(0.7,0){\arDU{\,\bullet}}
  \put(1.4,0){\arDU{\!\!\boxed{G_{30}}}}
  }
  \put(1.4,3.4){
  \put(0,0){\arDU{\,\bullet}}
  \put(0.7,0){\arDU{\!F_{-20}^\ast}}
  }
  \put(1.75,3.75){
  \put(0,0){\arDU{\!\!\boxed{G_{3,-2}}}}
  }
  }
    \put(2.5,0){
\put(-.35,2.05){
\put(0,0){\arUD{\,}}
\put(0.7,0){\arUD{\,}}
\put(1.4,0){\arUD{\,}}
\put(2.1,0){\arUD{\,}}
\put(2.8,0){\arUD{\,}}
\put(3.5,0){$\nearrow$}
\put(4.2,-.35){$\searrow$}
}
\put(0,1.7){
\put(0,0){\arUD{\,\bullet}}
\put(0.7,0){\arUD{\,\bullet}}
\put(1.4,0){\arUD{\!F_{02}^\ast}}
\put(2.1,0){\arUD{\,\bullet}}
\put(2.8,0){\arUD{\,\bullet}}
\put(3.5,0){$\nearrow$}
}
\put(.35,1.35){  \put(0,0){\arUD{\,\bullet}}
  \put(0.7,0){\arUD{\,\bullet}}
  \put(1.4,0){\arUD{\!F_{04}^\ast}}
  \put(2.1,0){\arUD{\,\bullet}}
  \put(2.8,0){\arUD{\,\bullet}}
  }
  \put(.7,1){
  \put(0,0){\arUD{\,\bullet}}
\put(0.7,0){\arUD{\,\bullet}}
\put(1.4,0){\arUD{\!F_{24}^\ast}}
\put(2.1,0){\arUD{\,\bullet}}
  }
  \put(1.05,0.65){
  \put(0,0){\arUD{\,\bullet}}
\put(0.7,0){\arUD{\,\bullet}}
\put(1.4,0){\arUD{\!\!\boxed{G_{04}}}}
  }
  \put(1.4,0.3){
  \put(0,0){\arUD{\,\bullet}}
\put(0.7,0){\arUD{\!E_{-20}^\ast}}
}
  \put(1.75,-.05){
  \put(0,0){\arUD{\!\!\boxed{G_{24}}}}
}
}
%
  \put(-.3,-2){
  \put(0,2){
  \put(0,0){\arDU{\!\!\boxed{\!F_{-20}}\!\text{- -}}}
  \put(0.7,0){\arDU{\!\!\!\!F_{-11} \text{ - - -}}}
  \put(1.4,0){\arDU{\!\boxed{F_{02}}\!\text{- - -}}}
  \put(2.1,0){\arDU{F_{13}\! \ \text{- - -}}}
  \put(2.8,0){\arDU{\!\boxed{F_{24}}}}
  }
  \put(0.35,2.35){
  \put(0,0){\arDU{\,\bullet}}
  \put(0.7,0){\arDU{\,\bullet}}
  \put(1.4,0){\arDU{\,\bullet}}
  \put(2.1,0){\arDU{\,\bullet}}
  }
  \put(0.7,2.7){
  \put(0,0){\arDU{\,\bullet}}
  \put(0.7,0){\arDU{\,\bullet}}
  \put(1.4,0){\arDU{\!\boxed{F_{04}}}}
  }
  \put(1.05,3.05){
  \put(0,0){\arDU{\,\bullet}}
  \put(0.7,0){\arDU{\,\bullet}}
  }
  \put(1.4,3.4){
  \put(0,0){\arDU{\,\bullet}}
  }
  }
   }
\end{picture}}
}}
\caption{Auslander-Reiten quiver for $\cC_{can}(Z_2(Z_\infty))$, Example \ref{examples of Z2(Zinfty)}.}\label{AR}
\label{fig02}
\end{center}
\end{figure}

\subsection{Geometric triangulations}\label{triangulation}

As pointed out in the Remark \ref{rem} some clusters will  correspond to triangulations of the unit disk $D^2$ after the limit points $L(Z)$ of the set $Z$ are removed. However, not all clusters are such. In this subsection we address this question. 
First we define  simplicial complex $K(\cS)$ associated to a cluster $\cS$ and investigate when there is a homeomorphism $\psi_{\cS}: |K(\cS)|\to D^2\backslash L(Z)$, i.e. when $(K(\cS), \psi_{\cS})$ defines a triangulation of $D^2\backslash L(Z)$.
First, we recall the definition and some basic properties of triangulations of topological spaces.

\begin{defn}
A \emph{triangulation} of a topological space $B$ is a simplicial complex $K$ together with a homeomorphism $\psi:|K|\cong B$ from the geometric realization
\[
	|K|=\coprod_{\sigma\in K^n}\Delta^n/\sim
\]
of $K$ to $B$. We recall that a \emph{simplicial complex} is a collection $K$ of finite nonempty subsets having the property that any nonempty subset of an element of $K$ is also an element of $K$. $K_n$ denotes the collection of $n+1$ element sets in $K$. Thus $K=\coprod_{n\ge0}K_n$ and all elements of $K$ are subsets of $K_0$. The maximum $n$ for which $K_n$ is nonempty is the \emph{dimension} of $K$.
\end{defn}

\begin{defn}
Given any cluster $\cS$ in $\cC_{can}(Z)$, the corresponding simplicial complex $K(\cS)$ is the 2-dimensional complex defined as follows.
\begin{enumerate}
\item[(0)] The vertex set is $K_0=Z$. 
\item Two points $x,y$ form an element $\{x,y\}\in K_1$ if either 
\begin{enumerate}
\item $E(x,y)$ lies in $\cS$, or
\item $x,y$ are consecutive elements of $Z$, i.e., either $y=x^+$ or $y=x^-$.
\end{enumerate}
\item Three points $x,y,z\in Z$ form an element of $K_2$ if and only if $\{x,y\},\{x,z\},\{y,z\}$ are elements of $K_1$.
\end{enumerate}
\end{defn}

\begin{defn} Let $\mathcal S$ be a cluster in the category $\mathcal C_{can}(Z)$ where $Z$ is an admissible subset of $S^1$.
Define the map
\[
	\psi_\cS:|K(\cS)|\to D^2
\]
by sending each vertex to itself, each 1-simplex of the form $\{z,z^+\}$ to the arc on the circle from $z$ to $z^+$ and any other 1-simplex $\{x,y\}$ to the closed geodesic connecting $x$ and $y$ (the circle or straight line which meets $S^1$ orthogonally at those two points) and, finally, the 2-simplices should be mapped by an arbitrary homeomorphism (which agrees with the already given map on the boundary) onto the closed region in $D^2$ enclosed by the boundary. 
\end{defn}

\begin{defn} A cluster $\cS$ is called a \emph{triangulation cluster} in the category $\cC_{can}(Z)$ if the pair $(K(\cS), \psi_{\cS})$ defines a triangulation of $D^2\backslash L(Z)$.
\end{defn}

It will be shown in Lemma \ref{1-1} that a necessary condition for a cluster to be a triangulation cluster, is that the cluster is locally finite. However, this is not a sufficient condition as will be shown in the example of Figure \ref{not2}.

\begin{defn}
We say that the cluster $\cS$ is \emph{locally finite} if every element of $Z$ occurs only finitely many times as an endpoint of an element of $\cS$.
\end{defn}

\begin{lem}\label{1-1} Let $\cS$ be a cluster in the triangulated category $\cC_{can}(Z)$ for an admissible subset $Z$ in $S^1$. Then:
\begin{enumerate}
\item The mapping $\psi_\cS$ is continuous and 1-1. 
\item $\psi_\cS$ is a homeomorphism onto an open subset of $D^2$ if and only if $\cS$ is locally finite.
\end{enumerate}
\end{lem}

\begin{proof} (1) The map $\psi_\cS$ is continuous since it is continuous on each simplex. To see that the map $\psi_\cS$ is 1-1, notice that since the elements of $\cS$ are compatible, the corresponding geodesics do not cross except possibly at endpoints by Remark \ref{compatible/noncrossing}. Therefore $\psi_\cS$ is 1-1 on the 1-skeleton of $|K(\cS)|$. The interior of each 2-simplex goes to the region enclosed by the three geodesics which are the images of the three sides of the 2-simplex. These regions cannot meet other geodesics since the geodesics do not cross. So, they are also disjoint making the entire mapping 1-1. 

(2) Now assume that $\cS$ is locally finite and take any point $x$ in the image of $\psi_\cS$. If $x$ is a vertex then there are finitely many geodesics at $x$ plus the two arcs connecting $x$ to $x^+$ and $x^-$. Given any two consecutive edges, the other endpoints of these edges are two points $y,z\in Z$ with the property that the object $E(y,z)$ is compatible with every object in $\cS$. So $\cS$ must contain an object isomorphic to $E(y,z)$. Take the closed region bounded by the geodesics or arcs connecting $x,y,z$ and delete the closed geodesic from $y$ to $z$. Let $U(x)$ be the union of these. ($U(x)$ is called the `open star' of $x$.) Then $U(x)$ is an open subset of $D^2$ containing $x$. Since $U(x)$ contains the interior of every simplex with $x$ as a vertex, the union of all $U(x)$ is equal the image of $\psi_\cS$. Therefore, the image of $\psi_\cS$ is open. Since any open subset of a locally compact space is locally compact, this implies that the image of $\psi_\cS$ is locally compact. Since $\cS$ is locally finite, $|K(\cS)|$ is also locally compact. So, we have a continuous bijection between two locally compact spaces. To see that it is a homeomorphism it suffices to show that this mapping is proper, i.e., that the inverse image of any compact subset is compact. But any compact subset is covered by a finite number of the open stars $U(x)$. The inverse image is therefore contained in a finite number of simplices. The closure of such a set is compact. Therefore the map is proper.

Conversely, suppose that $\psi_\cS$ is a homeomorphism onto an open subset of $D^2$. Then $|K(\cS)|$ is locally compact and this is only possible if $\cS$ is locally finite. So, this condition is necessary and sufficient.
\end{proof}

While the condition on a cluster of being locally finite does not guarantee that the cluster is a triangulation cluster, it is an easy test to show that a cluster is not a triangulation cluster. This will follow from the following lemma and Theorem \ref{triangulationtheorem}.

\begin{lem} \label{2implies1} Let $\cS$ be a cluster in $\cC_{can}(Z)$ for admissible $Z\subset S^1$. Suppose that for every sequence of objects $E_i(x_i,y_i)$ the points $x_i$ converge to a point $w\in L(Z)$ if and only if the points $y_i$ converge to the same point $w\in L(Z)$. Then $\cS$ is locally finite.
\end{lem}
\begin{proof} Suppose that $\cS$ is not locally finite. 
Then there is a point $z\in Z$ and an infinite sequence of objects $E_i(z,y_i)$ in $\cS$. Since $S^1$ is compact, there is an infinite subsequence $\{y'_j\}$ of $\{y_i\}$ which converges to some point $w\in L(Z)$. By the assumption of the lemma we have $z=w$. This gives a contradiction since $Z$ is discrete and hence $Z\cap L(Z)=\emptyset$.
\end{proof}

\begin{thm}\label{triangulationtheorem}
Let $\cS$ be a cluster in $\cC_{can}(Z)$ for admissible $Z\subset S^1$. Then the following statements are equivalent:
\begin{enumerate}
\item The cluster $\cS$ is a triangulation cluster.
\item  For every sequence of objects $E_i(x_i,y_i)$ in $\cS$ the points $x_i$ converge to a point $w\in L(Z)$ if and only if the points $y_i$ converge to the same point $w\in L(Z)$.
\end{enumerate}
\end{thm}

\begin{proof} 
(1)$\implies$(2) Suppose  $\cS$ is a triangulation cluster, i.e. it gives a triangulation of the topological space $D^2\backslash L(Z)$. 
Suppose (2) fails. The there exists a sequence of objects $E_i(x_i,y_i)$ such that $x_i$ converge to $w\in L(Z)\subset S^1$ and $y_i$ converge to a different point $w'\in L(Z)\subset S^1$. The points $w,w'$ cut $S^1$ into two components. Pick two elements $a,b$, one from each component so that $a,b$ are not in the set $Z$. Pick a path $P$ from $a$ to $b$. Then $P$ is a compact subset of $D^2\backslash L(Z)$ and therefore meets only finitely many simplices of the triangulation of $D^2\backslash L(Z)$. 
This is a contradiction since there is an infinite sequence of edges in the triangulation which goes from points close to $w$ to points close to $w'$ and will therefore cross the path $P$. So, Condition (2) must hold.

(2)$\implies$(1) Suppose that $\cS$ is a cluster in $\cC_{can}(Z)$ satisfying Condition (2). Then Lemma \ref{2implies1} implies that $\cS$ is locally finite.  By the Lemma \ref{1-1}(2)we know that $\psi_\cS:|K(\cS)|\to D^2$ is a homeomorphism onto its image. So, it suffices to show that the image is the complement of the limit set $L(Z)$.

Let $x$ be an element of $D^2\backslash L(Z)$. If $x\in S^1$ then either $x\in Z$ or $x$ lies between two consecutive elements of $Z$. So, $x$ is in the image of $\psi_\cS$. Suppose that $x$ lies in the interior of $D^2$. If $x$ lies on a 1-simplex, then $x$ lies in the image of $\psi_\cS$. So suppose $x$ does not lie in the image of any 1-simplex. Choose a geodesic $\gamma$ from $x$ to any point $z\in Z$. Then $\gamma$ meets at most finitely many 1-simplices in the triangulation since, otherwise, there would be accumulation points on two sides of $\gamma$ which are supposed to be equal, giving a contradiction. If $\gamma$ does not meet any of the 1-simplices of the triangulation then $\gamma$ lies between two consecutive 1-simplices $\alpha,\beta$ incident to the vertex $z$. Therefore, $x$ lies in the unique 2-simplex having $\alpha$ and $\beta$ on its boundary. If $\gamma$ meets a 1-simplex $\alpha$ then $x$ lies in one of the two 2-simplexes containing $\alpha$ in its boundary. So, again $x$ lies in the image of $\psi_\cS$. Therefore, $\psi_\cS$ has image $D^2\backslash L(Z)$. Therefore, Condition (2) implies that $\cS$ is a triangulation cluster. 
\end{proof}

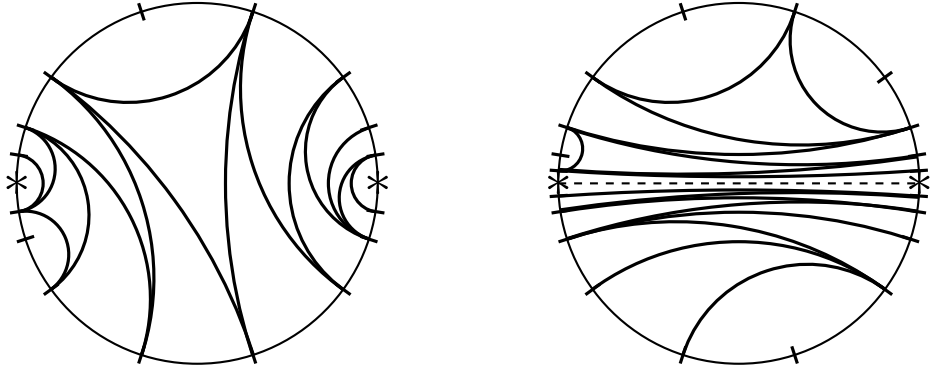
\begin{figure}[htbp]
\begin{center}
\begin{tikzpicture}[scale=.8]
\begin{scope}
\draw[thick] (0,0) circle [radius=3cm];
\draw (3,0) node{\huge$\ast$};
\draw (-3,0) node{\huge$\ast$};
\foreach \x in {-4,4,184,18,-18,9,-9,36,72,-72,144,108,189,198, 171,162,216,252,324,176}
\draw[very thick,rotate=\x] (2.85,0)--(3.15,0);
\clip (0,0) circle [radius=3cm];
\cluster{45}{3.37}{1.53};
\cluster{81}{6.61}{5.89};
\cluster{90}{9.71}{9.23};
\cluster{85.5}{12.85}{12.5};
\cluster{-85.5}{12.85}{12.5};
\ncluster{90}{9.71}{9.23};
\ncluster{81}{6.61}{5.89};
\ncluster{90}{5.1}{4.13};
\cluster{92}{26.5}{26.33};
\ncluster{92}{26.5}{26.33};
\cluster{169}{3.02}{.37};
\ncluster{90}{19.18}{18.94};
\cluster{90}{43.01}{42.9};
\cluster{-90}{43.01}{42.9};
\draw[very thick, rotate=108] (3.708,0) circle [radius=2.18cm];
\draw[very thick, rotate=108] (-3.708,0) circle [radius=2.18cm];
\draw[thick,dashed] (-3,0)--(3,0);
\end{scope}
\begin{scope}[xshift=-9cm] 
\draw[thick] (0,0) circle [radius=3cm];
\draw (3,0) node{\huge$\ast$};
\draw (-3,0) node{\huge$\ast$};
\foreach \x in {18,-18,9,-9,36,72,-72,144,108,189,198, 171,162,216,252,324}
\draw[very thick,rotate=\x] (2.85,0)--(3.15,0);
\clip (0,0) circle [radius=3cm];
\cluster{198}{5.1}{-4.13};
\ncluster{198}{5.1}{-4.13};
\cluster{9}{3.37}{1.53};
\ncluster{9}{3.37}{1.53};
\cluster{207}{4.24}{-3};
\cluster{175.5}{3.09}{-.72};
\cluster{-4.5}{3.09}{-.72};
\cluster{0}{3.15}{.97};
\cluster{0}{3.04}{.48};
\cluster{180}{3.04}{.48};
\ncluster{22.5}{3.09}{.72};
\draw[very thick, rotate=108] (3.708,0) circle [radius=2.18cm];
\draw[very thick] (3.708,0) circle [radius=2.18cm];
\draw[very thick] (9.7,0) circle [radius=9.23cm];
\draw[very thick, rotate=-144] (9.7,0) circle [radius=9.23cm];
\end{scope}
\end{tikzpicture}
\caption{Two locally finite clusters on $Z=Z_2(Z_{\infty})$: the one on the left satisfies Condition (2) and thus gives a geometric triangulation of $D^2\backslash L(Z)$, the one on the right does not satisfy Condition (2) and has geodesics accumulating to the horizontal line which is disjoint from all the simplices, hence simplices do not give a geometric triangulation of $D^2\backslash L(Z)$.}
\label{Figure88}\label{not2}
\end{center}
\end{figure}

\begin{rem}
Note that triangulation clusters are closed under mutations. Also locally finite clusters are closed under mutations. Therefore there are three cluster structures on the triangulated category $\cC_{can}(Z)$ given by either the triangulation clusters or the locally finite clusters or all clusters.
\end{rem}

\begin{prop}
Any triangulation of $D^2\backslash L(Z)$ with vertices on the set $Z$ is given by the map $\psi_\cS:|K(S)|\to D^2\backslash L(Z)$ for some triangulation cluster $\cS$ in $\cC_{can}(Z)$.\qed
\end{prop}
\begin{proof}Let  $\cS$ consist of the objects  $E(x,y)$ for all 1-simplices $\{x,y\}$ which are not subsets of the boundary $S^1$ of $D^2$.
\end{proof}

It is not clear whether triangulation clusters exist in $\cC_{can}(Z)$, especially when the set $Z$ is discrete with infinitely many limit points. We will show that triangulation clusters exist if the set of limit points is not too pathological. Let $L_n(Z), n\ge1,$ be defined recursively as follows: $L_1(Z)$ is the set of limit points of  $Z$ and $L_{n+1}(Z)$ is the set of limit points of $L_n(Z)$. Note that $L_n(Z)\subseteq L_m(Z)$ for all $1\le m\le n$.

\begin{thm}
Suppose that $Z\subset S^1$ is an admissible subset (Definition \ref{def: admissible subsets of S1}). Suppose also that $L_n(Z)$ is empty for sufficiently large $n$. Then $\cC_{can}(Z)$ contains a triangulation cluster.
\end{thm}

\begin{proof} For each limit point $w$ of $Z$ we will find a sequence of objects $E(x_i(w),y_i(w))$ of $\cC_{can}(Z)$ so that $x_i(w)$ converges to $w$ from one side and $y_i(w)$ converges to $w$ from the other side and so that 
all geodesics determined by $(x_i(w),y_i(w))$ are pairwise noncrossing. By Zorn's Lemma, this set is contained in a maximal noncrossing set $\cS$ of objects of $\cC_{can}(Z)$. We claim that any such set satisfies Condition (2) of the above theorem and therefore forms a triangulation cluster in $\cC_{can}(Z)$ and this cluster gives a triangulation of $D^2\backslash L(Z)$. The reason is simple. If $\cS$ contains a sequence of objects $E(a_i,b_i)$ with one end converging to a point $w\in L(Z)$ then 
the corresponding geodesics will cross infinitely many of the 
geodesics through $(x_j(w),y_j(w))$ unless the other end also converges to the same point $w$. Thus it suffices to find a sequence of noncrossing pairs $(x_i(w),y_i(w))$ converging to each $w\in L(Z)$.

Let $n$ be minimal so that $L_n(Z)$ is nonempty. Then $L_n(Z)$, being closed and discrete, is a finite set. So, we can find disjoint sequences $(x_i(w),y_i(w))$ converging to each $w\in L_n(Z)$. Let $m< n$ and suppose by downward induction that we have the desired collection of noncrossing pairs converging to every point in $L_{m+1}(Z)$.

Next, we look at $L_{m}(Z)$ as a closed discrete subset of the locally compact space $B_m=S^1\backslash L_{m+1}(Z)$. Choose a metric on $B_m$ so that $B_m$ is complete (send $L_{m+1}(Z)$ to `infinity') and the distance between any two elements of $L_{m}(Z)$ is at least 1. There are only finitely many of the already chosen pairs $(x_i(w),y_i(w))$ in a neighborhood of each point in $L_{m}(Z)$. So, we can choose a sequence of disjoint pairs $(x_i(w),y_i(w))$ converging to each $w\in L_{m}(Z)$. When we reach $m=0$ then we have the desired collection of noncrossing pairs converging to every limit point of $Z$ and we can complete this to a cluster which gives a triangulation of $D^2\backslash L(Z)$ as claimed.
\end{proof}

We extract the key point of the proof:

\begin{cor}\label{cor: each limit point of triangulation cluster}
A maximal set $\cS$ of noncrossing objects of $\cC_{can}(Z)$ is a triangulation cluster if and only if, for each limit point $w$ of $Z$ there is a sequence of objects $E(x_i(w),y_i(w))\in \cS$ so that $x_i(w)$ converges to $w$ from one side and $y_i(w)$ converges to $w$ from the other side.
\end{cor}

\begin{proof}
The proof of the theorem implies that this condition is sufficient for $\cS$ to be a triangulation cluster. 
Let $\cS$ be a triangulation cluster. Let $w$ be a limit point of $Z$.
Conversely, for any triangulation cluster $\cS$ and for every limit point $w$ of $Z$, consider a sequence of objects E$(x_i,y_i)$ satisfying the description above. Each object will cross only a finite number of edges of the triangulations since compact subsets of any locally compact cell complex will meet only finitely many cells. Let $(a_1,b_1),\cdots,(a_n,b_n)$ be the edges which meet $(x_i,y_i)$ ordered so that $x_i<a_1\le a_2\le\cdots\le a_n<y_i$. We extend the notation by letting $a_0=x_i$ and $a_{k+1}=y_i$. Then, for each $0\le k\le n$, $(a_k,a_{k+1})$ must be an object of $\cS$ since the curve that goes up from $a_k$ to the geodesic from $x_i$ to $y_i$, moves along the geodesic then back down to $a_{k+1}$ does not meet any object of $\cS$. One of these objects $(a_k,a_{k+1})$ must contain $w$ in its interior and the set of all these forms the desired configuration.
\end{proof}

\begin{prop}
In the cluster category $\cC_{id}(Z_\infty(Z))$ the Auslander-Reiten quiver is a union of components $C_{xy}$ where $x,y$ are nonlimit points of $Z$. The components $C_{xx}$ are of type $\ZZ A_\infty$ and the other components are of type $\ZZ A_\infty^\infty$. A triangulation cluster $\cS$ has infinitely many objects in $C_{xy}$ if $x,y$ are consecutive elements of $Z$. But $\cS\cap C_{xy}$ is finite for $y\neq x,x^-,x^+$. $\cS\cap C_{xx}$ can be finite or infinite.
\end{prop}

\begin{proof}
$C_{xy}$ is the collection of all objects $(a,b)$ where $x<a<x^+$ and $y<b<y^+$. It is clear that each $C_{xx}$ has type $\ZZ A_\infty$ and each $C_{xy}$ for $x\neq y$ has type $\ZZ A_\infty^\infty$.

Take a triangulation cluster $\cS$. If $C_{xy}$ contains an infinite number of objects of $\cS$ then there would be an infinite subset containing a sequence with one end converging to $x$ or $x^+$ and the other end converging to $y$ or $y^+$. By the theorem the two limits must be equal and this is possible only if $y=x,x^+$, or $x^-$. Thus $C_{xy}$ must be finite in all other cases.

For $y=x^-$, the corollary above implies that any triangulation cluster $\cS$ contains an infinite sequence of objects converging to $x$ from both sides. This is an infinite set in $C_{xx^-}$ for every $x\in Z$. Finally, for $C_{xx}$, the example in \cite[Fig 1]{IT15a} shows that $\cS\cap C_{xx}$ can be empty for all $x$ and Figure \ref{not2} (right side) gives an example showing that $\cS\cap C_{xx}$ (consisting of the $S_i$s and $T_j$s) can be infinite for all $x$.
\end{proof}

\subsection{Non-triangulation clusters and `cactus' cyclic posets} \label{non-triangulation}It follows from Theorem \ref{triangulationtheorem} that if a cluster $\cS$ is not a triangulation cluster then there exist a sequence of objects $E(x_i,y_i)$ in $\cS$ such that $x_i$ converge to a point $w\in S^1$ and $y_i$ converge to a point $w'\in S^1$ where $w\neq w'$. In this section we use certain equivalence relations in order to define new cyclic posets such that locally finite clusters correspond  to triangulation clusters.

Notice that if the admissible subset is finite, there are no limit points and therefore all clusters are triangulation clusters. So, for this subsection, we will only consider infinite subsets of $S^1$.
Due to the fact that non-triangulation clusters are  difficult to deal with, we consider only the case when the number of limit points of the admissible set is finite.
Recall from the Remark \ref{L(Zinfty)}(2) that in the case of infinite discrete admissible subset of $S^1$ with finitely many limit points, the set is of the form $Z(Z_{\infty})$ where $Z$ is the finite set of all limit points of $Z(Z_{\infty})$.

\begin{rem} \label{Vk is isomorphic to admissible} Let $Z(Z_{\infty}) \subset S^1$ be an  admissible subset, where $Z$ is finite. Consider $Z$ as a cyclic poset. For each $z\in Z$ let $z^+\in Z$ be the successor element in $Z$. Then:
\begin{enumerate}
\item The sets $Z$ and $Z(Z_{\infty})$ are disjoint subsets of $S^1$.
\item We define \emph{$Z_\infty$-interval} to be the set of points in $Z(Z_\infty)$ between $z$ and $z^+$. Note that $Z(Z_\infty)$ is the disjoint union of these $Z_{\infty}$-intervals, one for each $z\in Z$.  Also, $z$ is the descending limit of the points in the $Z_{\infty}$-interval corresponding to $z$.
\end{enumerate}
\end{rem}

\begin{rem}\label{modification} We will sometimes (e.g. in the proof of Lemma \ref{cactus first lemma}) need to modify the admissible set $Z(Z_\infty)\cong Z\ast\ZZ$ by deleting one or more $Z_\infty$-intervals. In that case, the resulting cyclic poset, call it $V$, will still be isomorphic to an admissible set, namely $V\cong Z_V\ast\ZZ$ where $Z_V$ is the set of all $z\in Z$ which are descending limits of elements of $V$. However, as a subset of $S^1$, such sets $V$ will not be admissible. For example, if $V$ is obtained by deleting from $Z(Z_\infty)$ the $Z_\infty$-interval between $z$ and $z^+$, as shown in Figure \ref{Fig: Pinch on boundary}, $V\cong Z_V\ast\ZZ$ where $Z_V=Z\backslash\{z\}$ but the points $z$ and $z^+$ will both be one-sided limit points of $V$. We remedy this with a topological trick: We collapse to a point the set $I(z,z^+)$ of all points in $S^1$ from $z$ to $z^+$ (shown in blue in Figure \ref{Fig: Pinch on boundary}). This means identify all of these point to one point and take the quotient topology on the result. This quotient space is still homeomorphic to a circle and the image of $V$ will be an admissible subset since $z$ and $z^+$ will be identified to one point.

If we take the closed $2$-disk $D^2$ with boundary the circle $S^1$ and identify $I(z,z^+)$ to one point we will get a space homeomorphic to $D^2$. We can do this to each $Z_\infty$-interval being deleted from the set $Z(Z_\infty)$ by iterating this process. 
\end{rem}

\begin{figure}[htbp]
\begin{center}
\begin{tikzpicture}
\begin{scope}[rotate=-90] 
\draw[thick] (0,0) circle [radius=2cm];
\begin{scope}
\foreach \x in {2,4, 10, 20, 35, 55, 70, 80,86,88}
\draw[thick,rotate=\x] (1.85,0)--(2.15,0);
\end{scope}
\begin{scope}[rotate=-90]
\foreach \x in {2,4, 10, 20, 35, 55, 70, 80,86,88}
\draw[thick,rotate=\x] (1.85,0)--(2.15,0);
\end{scope}
\begin{scope}[rotate=-180]
\foreach \x in {2,4, 10, 20, 35, 55, 70, 80,86,88}
\draw[thick,rotate=\x] (1.85,0)--(2.15,0);
\end{scope}
\end{scope}
\begin{scope}
\draw[very thick, color=blue] (1.8,0)--(2.2,0) (2.4,0) node{$z_1$};
\end{scope}
\begin{scope}[rotate=180]
\draw[very thick, color=red!90!black] (1.8,0)--(2.2,0) (2.4,0) node{$z_3$};
\end{scope}
\begin{scope}[rotate=-90]
\draw[very thick, color=red!90!black] (1.8,0)--(2.2,0) (2.4,0) node{$z_4$};
\end{scope}
\begin{scope}[rotate=90]
\draw[very thick, color=blue] (1.8,0)--(2.2,0) (2.4,0) node{$z_2=z_1^+$};
\end{scope}
\begin{scope}[rotate=45]
\clip (1,-1.4) rectangle (2.2,1.4);
\draw[very thick,color=blue] circle[radius=2cm];
\end{scope} 

\begin{scope} 
\draw (0,-3) node{$V\cong Z_V\ast\ZZ$ where $Z_V=\{z_2,z_3,z_4\}$};
\end{scope}


%
\draw[very thick, ->] (3,0)--(4,0) ;
\draw (3.5,.3) node{$p$};
\begin{scope}[xshift=7cm] 
\draw[thick] (0,0) circle [radius=2cm];

\begin{scope} 
\draw (0,-3) node{$p(V)$ is admissible.};
\end{scope}

\begin{scope}[rotate=45]
\foreach \x in {4, 7, 14, 26, 46, 68, 90, 111,123,130,133}
\draw[thick,rotate=\x] (1.85,0)--(2.15,0);
\end{scope}
\begin{scope}[rotate=45]
\draw[fill,color=blue] (2,0) circle[radius=3pt];
\end{scope}
\begin{scope}[rotate=-92]
\foreach \x in {4, 7, 14, 26, 46, 68, 90, 111,123,130,133}
\draw[thick,rotate=\x] (1.85,0)--(2.15,0);
\end{scope}
\begin{scope}[rotate=-180]
\foreach \x in {2,4, 10, 20, 35, 55, 70, 80,86,88}
\draw[thick,rotate=\x] (1.85,0)--(2.15,0);
\end{scope}
\begin{scope}[rotate=180]
\draw[very thick, color=red!90!black] (1.8,0)--(2.2,0) (2.4,0) node{$z_3$};
\end{scope}
\begin{scope}[rotate=-90]
\draw[very thick, color=red!90!black] (1.8,0)--(2.2,0) (2.4,0) node{$z_4$};
\end{scope}
\end{scope}
\end{tikzpicture}
\caption{When $Z_\infty$ intervals are deleted, we `pinch' the disk by collapsing to a separate point each portion of its boundary (shown in blue) where the $Z_\infty$-intervals have been removed.}
\label{Fig: Pinch on boundary}
\end{center}
\end{figure}
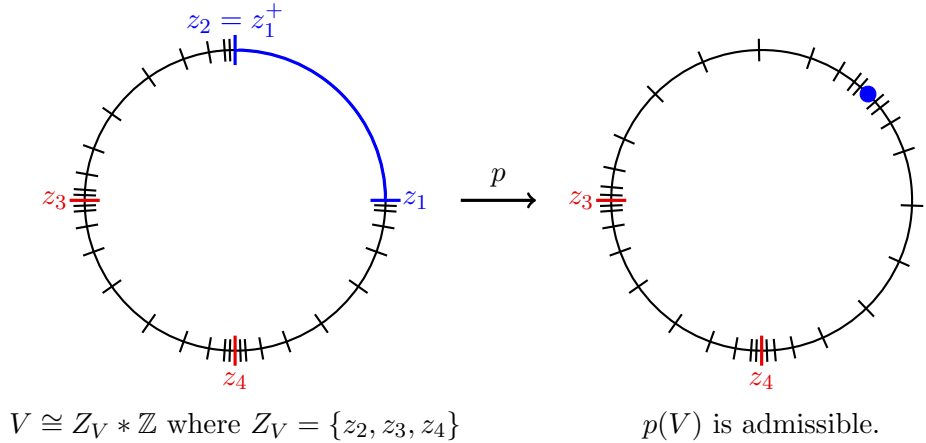

When a locally finite cluster is not a triangulation cluster, it gives a triangulation of an appropriate space: the ``cactus space'' $D_{\rho}$ (Definition \ref{cactus space}) minus a finite set. To construct the cactus space we introduce a noncrossing relation $\rho$ on $Z$ and the induced  $\rho$-noncrossing relation on the cyclic poset $Z(Z_\infty)$ (Definition \ref{rho on Z(Zinfty)}). With that, we define the cactus cyclic poset $Z(Z_\infty)/\rho$ (Definition \ref{cactus cyclic poset}), the cactus cluster category $\cC_{can}^{\rho}(Z(Z\infty))$ (Definition \ref{cactus cluster category}) and the cactus space $D_{\rho}$ (Definition \ref{cactus space}).

 \begin{defn}
Let $Z$ be a  subset of $S^1$. An equivalence relation $\rho$ on $Z$ is called \emph{noncrossing} if, whenever $a,b,c,d\in Z$ are in cyclic order and $a\!\sim_{\!\rho}\!c,\  b\!\sim_{\!\rho}\!d$, then\\
 $a\!\sim_{\!\rho}\!b\!\sim_{\!\rho}\!c\!\sim_{\!\rho}\!d$.
\end{defn}

\begin{defn}\label{rho on Z(Zinfty)}
Let $\rho$ be a noncrossing equivalence relation on  an admissible subset $Z$ of $S^1$. We say that two points $x,y\in Z(Z_\infty)$ are \emph{$\rho$-noncrossing} if 
the geodesic $\widehat{xy}$ does not cross any geodesic $\widehat{ww'}$ where  $w\!\sim_{\rho}\!w'$.

\end{defn}

\begin{eg}\label{eg: noncrossing}
In the figure below, $Z=\{a,b,c,d,e,f\}\subset S^1$ has noncrossing equivalence relation $\rho$ given by $a\sim d$, $d\sim e\sim f$ with $b$ in its own equivalence class. This equivalence relation is noncrossing since the four (dashed) geodesics $\widehat{ac},\widehat{de},\widehat{ef},\widehat{df}$ connecting equivalent points do not cross. An example of a ``crossing'' equivalence relation on $Z$ would be given by $a\sim c$ and $b\sim f$. 

\begin{center}
\begin{tikzpicture}[scale=.75]
\begin{scope}
\draw[thick] (0,0) circle [radius=3cm];
\foreach \x/\xtext in {5/z_2,65/y_3,115/x_1,130/x_2,175/x_3,-125/y_1,-110/y_2,-55/z_1}
\draw[thick, rotate=\x] (2.9,0)--(3.1,0) (3.5,0) node {$\xtext$};
\foreach \y/\ytext in {90/a,150/b,-150/c,-90/d,-30/e,30/f}
\draw[very thick, rotate=\y,fill]  
(3.5,0) node{$\ytext$} (3,0) circle[radius=2pt];
\clip (0,0) circle [radius=3cm];
\foreach \x in {0,-60}
\dcluster{\x}{3.46}{1.73};
\dcluster{-30}{6}{5.2};
\dcluster{150}{6}{5.2};
\draw[rotate=-30] (0.5,0) node{$\rho$};
\draw[rotate=150] (1.1,0) node{$\rho$};
\draw[rotate=0] (2.1,0) node{$\rho$};
\draw[rotate=-60] (2.1,0) node{$\rho$};
\end{scope}
\end{tikzpicture}
\end{center}

The points $x_1,x_2,x_3\in Z(Z_\infty)$ are pairwise $\rho$-noncrossing since the geodesics $\widehat{x_ix_j}$ do not cross the geodesics connecting $\rho$-equivalent points of $Z$. Similarly, the points $y_1,y_2,y_3$ are pairwise $\rho$-noncrossing. However, $z_1,z_2$ are $\rho$-crossing since $\widehat{z_1z_2}$ crosses the geodesic $\widehat{de}$.
\end{eg}

We also note that elements of the same $Z_\infty$-interval, such as $x_1,x_2$ and $y_1,y_2$ in Example \ref{eg: noncrossing} above, are $\rho$-noncrossing since there are no elements of $Z$ between them.

\begin{lem} Let  $\rho$ be a noncrossing equivalence relation on a finite subset $Z\subset S^1$. Then
$\rho$-noncrossing is an equivalence relation on $Z(Z_\infty)$ with finitely many equivalence classes.
\end{lem}

\begin{proof}
The (finitely many) geodesics connecting distinct equivalent points in $Z$ divide the disk $D^2$ into finitely many connected components.
These components are in bijection with the $\rho$-noncrossing equivalence classes in $Z(Z_\infty)$ since $x,y \in Z(Z_\infty)$ are $\rho$-noncrossing if and only if they lie in the same component of $D^2$. (Example \ref{eg: noncrossing} has 4 components.)
\end{proof}

\begin{defn}\label{cactus cyclic poset}
Let $Z\subset S^1$ be a finite subset of the circle. Let $\rho$ be a noncrossing equivalence relation on $Z$. Then we define the \emph{cactus cyclic poset} $Z(Z_\infty)/\rho$ to be the  cyclic poset with the same underlying set $Z(Z_\infty)$ but with the new cyclic cocycle $c_\rho$ given as: 
define the function $B:Z(Z_\infty)^2\to \NN$ by $B(x,y)=0$ if $x,y$ are $\rho$-noncrossing and $B(x,y)=1$ otherwise. Then the cyclic poset cocycle $c_\rho$ on $Z(Z_\infty)$ is given by
\[
	c_\rho(x,y,z)=c(x,y,z)+B(x,y)+B(y,z)-B(x,z).
\]
It is easy to see that this is a reduced cocycle on $Z(Z_\infty)$. 
\end{defn}

\begin{lem}\label{cactus first lemma} Let $\rho$ be a noncrossing equivalence relation on  a finite subset $Z$ of $S^1$.
\begin{enumerate}
\item Let $V_k$ be a $\rho$-noncrossing equivalence class in $Z(Z_\infty)/\rho$. Then the cluster category of $V_k$ embeds as a full subcategory of both $\cC_{can}(Z(Z_\infty))$ and $\cC_{can}(Z(Z_\infty)/\rho)$. 
\item If $X\in \cC_{can}(V_j)$ and $Y\in \cC_{can}(V_k)$ where $V_j,V_k$ are distinct $\rho$-noncrossing equivalence classes then $\Hom_\cC(X,Y)=0=\Hom_\cC(Y,X)$ where $\cC=\cC_{can}(Z(Z_\infty))$.
\end{enumerate}
\end{lem}

\begin{proof} (1) The subset $V_k\subset S^1$ can be modified to an admissible subset of $S^1$ by "pinching" $Z_{\infty}$ intervals of $Z(Z_\infty)$ which are in the other $\rho$-noncrossing equivalence classes, as done in the Remark \ref{modification}.
The cyclic poset $V_k$ is isomorphic to an admissible subset of $S^1$. So, the cluster category $\cC_{can}(V_k)$ is defined.
If $x,y,z\in V_k$ then $B(x,y),B(y,z),B(x,z)$ are all zero by definition. So, $c_\rho(x,y,z)=c(x,y,z)$ and the cyclic poset structure on $V_k\subseteq Z(Z_\infty)/\rho$ is the same as the one induced by the inclusion $V_k\subseteq Z(Z_\infty)$. Also, $z,z^+$ are $\rho$-noncrossing since there is no element of $Z$ between them. Therefore, each $\rho$-noncrossing equivalence class $V_k$ is invariant under the canonical automorphism of $Z(Z_\infty)$. This makes $\cC_{can}(V_k)$ a full subcategory of both $\cC_{can}(Z(Z_\infty))$ and $\cC_{can}(Z(Z_\infty)/\rho)$.

(2) To see the hom-orthogonality, let $x,y\in V_j$ and $a,b\in V_k$ where $V_j\neq V_k$. Then $\widehat{xy}$ and $\widehat{ab}$ do not meet even at their boundary points. So, $E(x,y)$ and $E(a,b)$ are hom-orthogonal in $\cC=\cC_{can}(Z(Z_\infty))$ (Lemma 5.1,  \cite{IT13}).
\end{proof}

This immediately implies the following.
\begin{defn}\label{cactus cluster category} We define \emph{cactus cluster category} to be $\cC_{can}^\rho(Z(Z_\infty))$, the additive full subcategory of  $\cC_{can}(Z(Z_\infty))$ with indecomposable objects $E(x,y)$ where $x,y$ are $\rho$-noncrossing.
\end{defn}
\begin{prop}\label{prop: cactus as full subcategory}
Let  $\cC_{can}^\rho(Z(Z_\infty))$ be the additive full subcategory of  $\cC_{can}(Z(Z_\infty))$ with indecomposable objects $E(x,y)$ where $x,y$ are $\rho$-noncrossing.
Let  $\{V_k\}$ be the $\rho$-noncrossing equivalence classes in $Z(Z_\infty)/\rho$. Then 
\[
	\cC_{can}^\rho(Z(Z_\infty))\cong \smallprod_{\!k} \ \cC_{can}(V_k).
\]
\end{prop}

\begin{lem}\label{cactus second lemma}
If $x,y$ are in different $\rho$-noncrossing equivalence classes then the object $E(x,y)$ is not an object in the cluster category of the cactus cyclic poset $Z(Z_\infty)/\rho$. Thus, the only indecomposable objects of $\cC_{can}(Z(Z_\infty)/\rho)$ are $E(x,y)$ where $x,y$ are in the same $\rho$-noncrossing equivalence class and $y\neq x,x^+,x^-$.
\end{lem}
\begin{proof} For any two distinct points in $S^1$ we have $c(x,y,x)=1$ since going from $x$ to $y$ back to $x$ always makes a full circle. On the other hand, $B(x,y)=B(y,x)=1$ by definition. So,
\[
	c_\rho(x,y,x)=c(x,y,x)+B(x,y)+B(y,x)-B(x,x)=1+1+1-0=3
\]
which means that, for any two morphisms $\alpha:P_x\to P_y$ and $\beta:P_y\to P_x$ the composition $\beta\circ \alpha:P_x\to P_x$ is divisible by $t^3$. In particular there is no matrix factorization of $t$ (a choice of $\alpha,\beta$ so that $\beta\circ\alpha=t$). So, $E(x,y)$ does not exist as an element of the cluster category of the cactus cyclic poset $Z(Z_\infty)/\rho$.

If $x, y\in Z(Z_\infty)/\rho$ are in the same $\rho$-noncrossing equivalence class and $y\neq x,x^+,x^-$ then $c_\rho(x,y,x)=c(x,y,x)=1$ and $E(x,y)$ is a nonzero object of $\cC_{can}(Z(Z_\infty)/\rho)$. Since we have already excluded the other possible objects, these are the only indecomposable objects of the cactus cluster category $\cC_{can}(Z(Z_\infty)/\rho)$.
\end{proof}

For the next theorem we recall that a finite product of cluster categories is a cluster category and that a cluster in a product is given by a cluster in each factor.

\begin{thm}\label{thm: cactus category= prod C(Vk)}
The cluster category $\cC_{can}(Z(Z_\infty)/\rho)$ is isomorphic to the full subcategory $\cC_{can}^\rho(Z(Z_\infty))$ from Proposition \ref{prop: cactus as full subcategory} and, therefore, also isomorphic to the product of cluster categories $\cC_{can}(V_k)$ where $V_k$ are the $\rho$-noncrossing equivalence classes in $Z(Z_\infty)/\rho$:
\[
	\cC_{can}(Z(Z_\infty)/\rho)\cong \cC_{can}^\rho(Z(Z_\infty))\cong \smallprod_{\!k} \cC_{can}(V_k).
\]
\end{thm}

\begin{proof}
By Lemma \ref{cactus second lemma} every object of $\cC_{can}(Z(Z_\infty)/\rho)$ is a direct sum $\bigoplus X_k$ where $X_k\in \cC_{can}(V_k)$. By Lemma \ref{cactus first lemma} $\cC_{can}(V_k)$ is embedded as a full subcategory of $\cC_{can}(Z(Z_\infty)/\rho)$. So, $\cC_{can}(Z(Z_\infty)/\rho)$ and $\cC_{can}^\rho(Z(Z_\infty))$ have the same set of objects. It only remains to show that $\Hom(X,Y)=0$ in $\cC_{can}(Z(Z_\infty)/\rho)$ if $X\in \cC_{can}(V_j)$ and $Y\in \cC_{can}(V_k)$ where $j\neq k$.

Let $E(x,y)\in \cC_{can}(V_j)$ and $E(a,b)\in \cC_{can}(V_j)$ where $j\neq k$. Then $\widehat{xy}, \widehat{ab}$ are noncrossing and, by symmetry, we may assume that $x,y,a,b$ are in cyclic order in $S^1$. This implies that
\[
	c_\rho(x,y,a)=c(x,y,a)+B(x,y)+B(y,a)-B(x,a)=0+0+1-1=0.
\]
Similarly $c_\rho(x,y,b)=0$. So, every morphism $P_x\to P_a\oplus P_b$ factors through $P_y$. So, any map $E(x,y)\to E(a,b)=P_a\oplus P_b$ factors through $E(y,y)=0$ showing that $\Hom(E(x,y),E(a,b))=0$ in $\cC_{can}(Z(Z_\infty)/\rho)$ just as it is in $\cC_{can}(Z(Z_\infty))$. So, $\cC_{can}(Z(Z_\infty)/\rho)$ embeds as a full subcategory of $\cC_{can}(Z(Z_\infty))$ as claimed.
\end{proof}

The purpose of the above equivalence considerations, was to deal with locally finite clusters which are not triangulation clusters in $\cC_{can}(Z(Z_\infty))$, i.e. do not define triangulation of $D^2\backslash Z$ (in this subsection we only deal with the case of discrete admissible sets which have only finitely many limit points, hence are of the form 
$Z(Z_\infty)$ with a finite $Z\subset S^1$.

\begin{defn} Let $\cS$ be a locally finite cluster in the category $\cC_{can}(Z(Z_\infty))$. Define an equivalence relation $\rho_{\cS}$ on the set $Z$ to be generated by $w\sim w'$ if there is a sequence of objects $E(x_i,y_i)$ in $\cS$ with $x_i$ converging to $w$ and $y_i$ converging to $w'$. 
\end{defn}

Let  $Z(Z_\infty)/\rho_{\cS}$  be the  associated cactus cyclic poset and let $\cC_{can} (Z(Z_\infty)/\rho_{\cS})$ be the associated triangulated category (see Figure \ref{FigZ10}). 

\begin{figure}[htbp]
\begin{center}
\begin{tikzpicture}
\begin{scope}
\foreach \x in {182, 185, 189, 194,  201, 206, 210, 213, 215, 217, 220, 224, 229, 235,242, 249, 256, 261, 265, 268,92, 95, 99, 104, 111, 116, 120, 123,127, 130, 134, 81, 85, 88,139, 144, 151, 159, 166,171, 175,178,-57, 57, -60,60, -64,64, -69,69,76, -76, -81, -85, -88, 37, 40, 45, 50,53,-37, -40, -45, -50,-53, 2,3,6,10, 25, 29, 32,33,-2,-3,-6,-10,-18,18, -25, -29, -32, -35,-33}
\draw[thick,rotate= \x] (3.2,0)--(3,0);
\begin{scope}
\clip (0,0) circle[radius=3cm];
\draw[fill,color=gray!20!white] (3.66,0) circle[radius=2.1cm];
\draw[fill,color=white, rotate=17.5] (3.15,0) circle[radius=.95cm];
\draw[fill,color=white, rotate=-17.5] (3.15,0) circle[radius=.95cm];
\end{scope}
\draw[rotate=235] (3.6,0) node{$D_1$};
\draw[rotate=108] (3.6,0) node{$D_3$};
\draw[rotate=-45] (3.6,0) node{$D_4$};
\draw[rotate=17] (3.6,0) node{$D_5$};
\draw[rotate=-17] (3.6,0) node{$D_6$};
\draw[rotate=215] (3.5,0) node{$w_1$};
\draw[rotate=270] (3.5,0) node{$w_3$};
\draw[rotate=180] (3.5,0) node{$w_2$};
\draw[rotate=125] (3.5,0) node{$w_4$};
\draw[rotate=90] (3.5,0) node{$w_5$};
\draw[rotate=55] (3.5,0) node{$w_6$};
\draw[rotate=-55] (3.5,0) node{$w_7$};
\draw[rotate=35] (3.5,0) node{$w_8$};
\draw[rotate=-35] (3.5,0) node{$w_9$};
\draw[rotate=0] (3.6,0) node{$w_{10}$};
\draw[rotate=0] (2,0) node{$X$};
\foreach\x in {-55, -35, 0, 35, 55, 90, 125, 180, 215,270}
\draw[thick,rotate= \x] (3.2,0)--(3,0);
\foreach\x in {-55, -35, 0, 35, 55, 90,125, 180, 215, 270}
\draw[fill,rotate= \x] (3,0) circle[radius=3pt];
\draw[thick] (0,0) circle[radius=3cm];
\clip (0,0) circle[radius=3cm];
\lcluster{200}{3.02}{.32};
\lcluster{202}{3.03}{.42};
\lcluster{213.5}{3.01}{.18};
\lcluster{215}{3.01}{.26};
\lcluster{217}{3.02}{.37};
\lcluster{219.5}{3.04}{.5};
\lcluster{211.5}{3.15}{.95};
\lcluster{214.5}{3.20}{1.12};
\lcluster{212}{3.26}{1.27};
\lcluster{219}{3.46}{1.73};
\lcluster{242}{3.02}{.37};

\lcluster{217}{3.54}{1.87};
\lcluster{220.5}{3.68}{2.14};
\lcluster{219}{3.76}{2.26};
\lcluster{221.5}{3.89}{2.47};
\lcluster{223.5}{4.01}{2.65};
\lcluster{225}{4.1}{2.8};

\lcluster{107.5}{3.11}{.83};
\lcluster{109}{3.09}{.75};
\lcluster{107.5}{3.07}{.67};
\lcluster{105.5}{3.05}{.56};
\lcluster{107.5}{3.03}{.45};
\lcluster{110}{3.02}{.32};

\lcluster{109.5}{3.3}{1.37};
\lcluster{107.5}{3.25}{1.24};
\lcluster{109}{3.21}{1.15};
\lcluster{107.5}{3.18}{1.06};
\lcluster{117.5}{4.01}{2.65};
\lcluster{120}{3.86}{2.43};
\lcluster{112.5}{3.52}{1.84};
\lcluster{151.5}{3.03}{.39};
\lcluster{110}{3.43}{1.66};
\lcluster{107.5}{3.35}{1.5};
\lcluster{42.5}{-5.44}{4.53};
\lcluster{45}{-5.1}{4.13};
\lcluster{43}{-4.87}{3.84};
\lcluster{45}{-4.67}{3.58};
\lcluster{43.5}{-4.53}{3.39};
\lcluster{45}{-4.4}{3.22};
\lcluster{121}{4.24}{3};
\lcluster{45}{-5.82}{4.99};
\lcluster{2.5}{7.52}{6.9};
\lcluster{-3.5}{9.98}{9.51};
\lcluster{-70}{3.02}{.32};
\lcluster{1.5}{5.74}{4.9};
\lcluster{-2}{6.39	}{5.64};
\lcluster{0}{5.51}{4.62};
\lcluster{0}{6}{5.2};
\lcluster{0}{6.84}{6.15};
\lcluster{0}{12.4}{12.03};
\lcluster{0}{3.76}{2.26};
\lcluster{0}{3.92}{2.52};
\lcluster{0}{4.67}{3.58};
\lcluster{0}{4.98}{3.98};
\lcluster{1.5}{3.83}{2.39};
\lcluster{2.5}{4.07}{2.75};
\lcluster{5}{4.24}{3};
\lcluster{2.5}{4.44}{3.27};
\lcluster{1.5}{4.82}{3.77};
\lcluster{17.5}{3.03}{.39};
\lcluster{15.5}{3.04}{.5};
\lcluster{17.5}{3.06}{.61};
\lcluster{16}{3.08}{.69};
\lcluster{17.5}{3.1}{.78};
\lcluster{18.25}{3.11}{.82};
\lcluster{17.5}{3.12}{.86};
\lcluster{-12}{3.02}{.32};
\lcluster{-23.5}{3.01}{.29};
\lcluster{-17.5}{3.06}{.61};
\lcluster{-16}{3.08}{.69};
\lcluster{-17.5}{3.1}{.78};
\lcluster{-18.25}{3.11}{.82};
\lcluster{-17.5}{3.12}{.86};

\dotcluster{0}{3.66}{2.1};
\dotcluster{17.5}{3.15}{.95};
\dotcluster{-17.5}{3.15}{.95};
%
\dotcluster{0}{5.23}{4.28};
\dotcluster{107.5}{3.15}{.95};
\dotcluster{225}{4.24}{3};
\end{scope}
\draw (-.4,.3) node{$D_2$};
\end{tikzpicture}
\caption{A locally finite cluster $\cS$ on $Z_{10}(Z_{\infty})$. Dark spots are the points of $Z_{10}$ which are the two-sided limit points of the  set $Z_{10}(Z_\infty)$. 
There are 5 equivalence classes of $\rho_{\cS}$; their sizes are 1,2,2,2,3.
Limits of arcs, shown as dotted curves, will be collapsed in the cactus space $D_{\cS}$. The gray region will also be identified to one point in $D_{\cS}$ (Figure \ref{fig04}).
}
\label{FigZ10}
\end{center}
\end{figure}
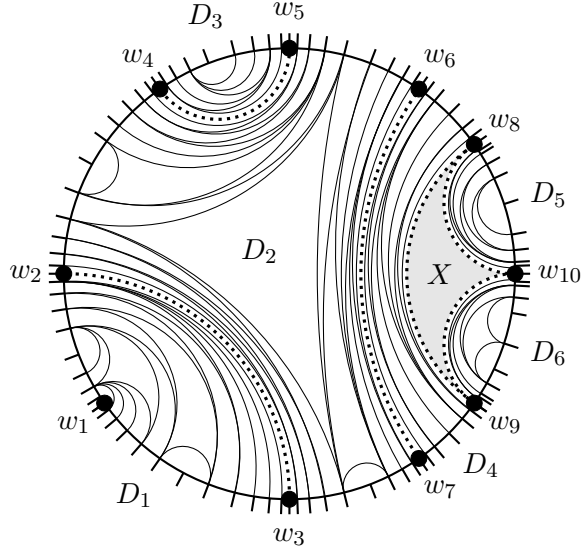
\begin{defn}\label{cactus space} We also construct  \emph{cactus space} $D_{\rho}$ associated to a noncrossing relation $\rho$ on a finite subset $Z\subset S^1$.
For each equivalence class $Z_k=\{w_1,\cdots,w_n\}$ in $Z$ we contract to a separate point $p_k$ the subset $W_k$ of the closed disk $D^2$ given as follows. When $n=1$, $W_k$ is one point $W_k=Z_k$. When $n=2$, $W_k$ is the closed geodesic $W_k=\widehat{w_1w_2}$. When $n\ge3$, we take $W_k$ to be the closed region in $D^2$ bounded by the closed geodesics $\widehat{w_iw_{i+1}}$, $i=1,\cdots,n$ where $w_0=w_n$. An example with $n=3$ is shown in gray in Figure \ref{FigZ10}.

For each $k$, the process of identifying the set $W_k$ to one point $p_k$ creates $n$ disks ($n-1$ additional disks) joined together at one shared point $p_k$. For the example in Figure \ref{FigZ10}, three of the equivalence classes $Z_k$ have size 2 and create one additional 2-disk each and on equivalence class has size 3 creating two additional 2-disks. The total is 5 additional 2-disks for a total of 6. See Figure \ref{fig04}.

This process gives a topological space $D_{\cS}$ which is homeomorphic to a finite union of 2-disks $\{D_i\}$ attached together at a finite set of points on their boundaries so that the union is simply connected. (Notice that any pair of 2-discs will be attached at, at most one point.) The image of $Z(Z_\infty)$ in each 2-disk $D_i$ is equal to $Z_i(Z_\infty)$, where $Z_i$ is  the image of $Z$ in $D_i$. The attaching points are in the image of $Z$. We call this the `cactus space' corresponding to the cluster $\cS$, or `cactus space' corresponding to the equivalence relation $\rho$. 
\end{defn}

Notice that in Figure \ref{FigZ10} limits of arcs, shown as dotted geodesics 
$\widehat{w_0w}$ for $w_0\!\sim_{\rho}\!w$ will be collapsed in the cactus space $D_{\cS}=D_{\rho}$ in Figure \ref{fig04}.
 The gray region $X$ is also identified to one point in Figure \ref{fig04}.

\begin{rem} Let $Z(Z_\infty)$ be the admissible cyclic poset with finite limit set $Z\subset S^1$. \begin{enumerate}
\item To each locally finite cluster $\cS$ we associate the cactus space $D_{\cS}$.
\item The locally finite cluster $\cS$ is a triangulation cluster if and only if the associated cactus space is equal to $D^2$.
\item To each noncrossing equivalence relation on $Z$ we associate cactus space $D_{\rho}$.
\item There are many locally finite clusters which define the same noncrossing equivalence relation $\rho$ and therefore the same cactus space $D_{\rho}$ which will be described precisely in Proposition \ref{ST}.
\end{enumerate}
\end{rem}
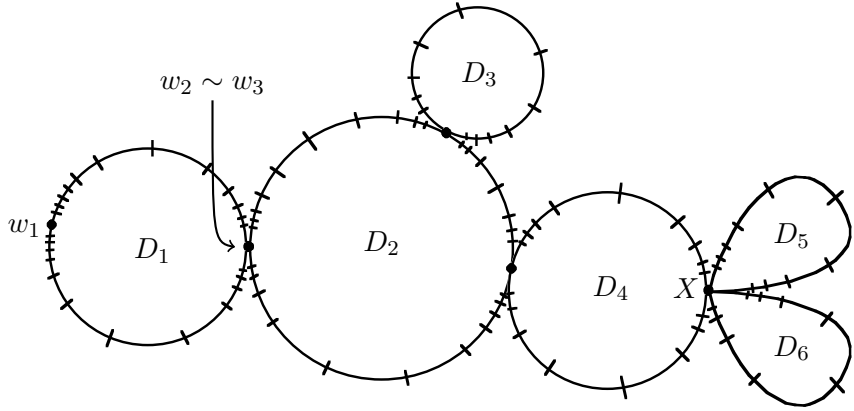
\begin{figure}[htbp]
\begin{center}
\begin{tikzpicture}
\clip (0.8,0) rectangle (13,5.6);
\pgfimage[width=5.5in]{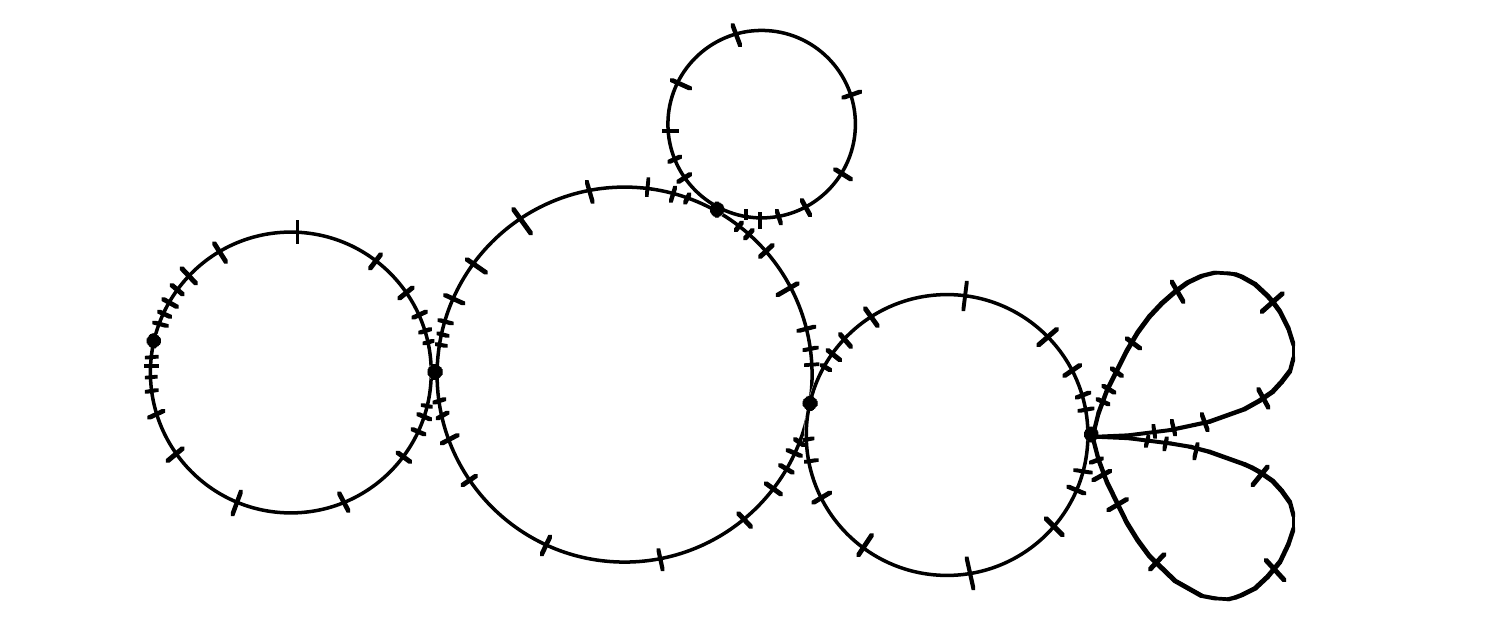};
\begin{scope}[xshift=-5in]
\draw (1.5,2.3) node{$D_1$};
\draw (-.2,2.6) node{$w_1$};
\draw (2.3,4.5) node{$w_2\sim w_3$};
\draw[->,thick] (2.3,4.3)..controls (2.3,2.5) and (2.3,2.4)..(2.6,2.35);
\draw (4.55,2.4) node{$D_2$};
\draw (5.85,4.65) node{$D_3$};
\draw (7.6,1.8) node{$D_4$};
\draw (8.6,1.8) node{$X$};
\draw (10,2.5) node{$D_5$};
\draw (10,1) node{$D_6$};
\end{scope}
\end{tikzpicture}
\caption{Cactus space $D_{\cS}$ for the locally finite cluster $\cS$ in the category $\cC_{can}(Z_{10}(Z_\infty))$
       from Figure \ref{FigZ10}. 
              The 5 equivalence classes of size 1,2,2,2,3 of Figure \ref{FigZ10} appear here as spots and the elements of $Z_{10}(Z_\infty)$ are the dashes.}\label{fig04}
\end{center}
\end{figure}
\begin{prop} \label{ST} Let $Z(Z_\infty)$  be an admissible cyclic poset, with finite limit set $Z\subset S^1$. Let $\rho$ be a noncrossing relation on $Z$. Then there is a 1-1 correspondence:
\[
	\left\{\cS |
	\begin{array}{c}\cS 
	\text{ locally finite cluster in }
	\\
	\text{ }\cC_{can}(Z(Z_\infty))
	\text{ with $\rho_{\cS} =\rho$}
	\end{array}
	\right\} 
	\leftrightarrow
	\left\{\cT |
	\begin{array}{c}\cT=\{\cT_1,\dots, \cT_k\}\\
	\cT_i \text{ triangulation cluster for disk } D_i \\
	\text{of the cactus space } D_\rho
	\end{array}
	\right\}.
\]
\end{prop}

\begin{proof}
This follows from Theorem \ref{thm: cactus category= prod C(Vk)}. Any cluster $\cS$ with associated noncrossing equivalence relation $\rho$ has objects in the $\rho$-noncrossing components $V_k$ of $Z(Z_\infty)$. Therefore $\cS$ is a cluster in the full subcategory $\cC_{can}^\rho(Z(Z_\infty))$ of $\cC_{can}(Z(Z_\infty))$. By Theorem \ref{thm: cactus category= prod C(Vk)}, the objects in each $V_k$ form a cluster $\cS_k$ in $\cC_{can}(V_k)$ which is still locally finite. We claim that $\cS_k$ is a triangulation cluster. Indeed, consider a sequence of objects $E(x_i,y_i)$ in $\cS_k$ with $x_i\to w$ and $y_i\to w'$. We claim that $w\sim w'$. Considered as objects in $\cC_{can}(Z(Z_\infty))$, these same objects $E(x_i,y_i)$ are in $\cS$ and the limit points of $z,z'$ of $\{x_i\}$ and $\{y_i\}$ are $\rho$-equivalent by definition of $\rho$. Therefore, these points, $z,z'$ map to the same point $w=w'$ in the cactus space by construction of that space. This shows that $\cS_k$ is a triangulation cluster.

Conversely, suppose that $\{\cS_k\}$ is a set of triangulation clusters, one for each $\rho$-noncrossing equivalence class $V_k$. Take the union $\cS=\bigcup \cS_k$. We claim that $\cS$ is a locally finite cluster in $\cC_{can}(Z(Z_\infty))$ with corresponding equivalence relation $\rho$. 

To determine the noncrossing equivalence relation, suppose that $E(x_i,y_i)$ is an infinite sequence of objects in $\cS$ where $x_i\to z$ and $y_i\to z'$. Since there are only finitely many $\cS_k$, one of them contains infinitely many of these objects and, passing to this subsequence, we may assume that all $E(x_i,y_i)$ lie in the same $\cS_k$. Since $\cS_k$ is a triangulation cluster, $\{x_i\}$ and $\{y_i\}$ converge to the same point $w$ in the cluster space. So, $z,z'$ both map to $w$ which implies that $z\sim_\rho z'$. Conversely, suppose that $z\sim_\rho z'$ where $z,z'$ are consecutive points in an equivalence class $W_j\subseteq Z\subset S^1$. Then $z,z'$ are one-sided limit points of some $V_k$ which map to the same point $w$ in the cactus space. By Corollary \ref{cor: each limit point of triangulation cluster}, there must exist objects $E(x_i,y_i)$ in $\cS_k$ where $\{x_i\}$ and $\{y_i\}$ converge to $w$ from opposite sides. Then in the full circle $x_i,y_i$ converge to $z$ and $z'$. So $z\sim z'$ in the equivalence relation corresponding to $\cS$.

The rest is very easy. We already know that $\cS$ is locally finite and objects are pairwise compatible. Also, $\cS$ is a maximal compatible set. Otherwise, there is an object $E(x,y)$ compatible with all objects of $\cS$ where $x,y$ are not $\rho$-noncrossing. This means the geodesic $\widehat{xy}$ crosses a geodesic $\widehat{zz'}$ where we may take $z,z'$ to be consecutive points in an equivalence class $W_j\subseteq Z$. As in the last paragraph, there is a sequence of objects $E(x_i,y_i)$ in some $\cS_k$ so that $\widehat{x_iy_i}$ converge to $\widehat{zz'}$. This implies that $\widehat{x_iy_i}$ will cross $\widehat{xy}$ for sufficiently large $i$ which contradicts that assumption that $E(x,y)$ is compatible with $\cS$. So, $\cS$ is a locally finite cluster having all the desired properties. This proves the theorem.
\end{proof}

The conclusion is that, when $Z$ has a finite number of limit points, the locally finite clusters come from triangulation clusters of smaller cluster categories. It would be interesting if there were examples where locally finite clusters do not exist or if there were locally finite clusters which do not come from triangulation clusters of smaller categories.

\section{Cluster structures of the cyclic posets of $S^1$ with automorphism $\varphi_{\theta}$} \label{phi=theta}

In this section we consider cyclic posets of $S^1$ with automorphisms, i.e. $(X,c)_{\varphi}$ from Section \ref{Cyclic posets with automorphisms}. Here the automorphism $\varphi=\varphi_\theta$ is different from the identity and from the canonical automorphism: these two cases were done in Sections \ref{phi=id} and \ref{phi=canonical}.

\subsection{Automorphism $\varphi_{\theta}$}

Let $X$ be a subset of $S^1$ which is invariant under rotation in either direction by a positive angle $\theta$. Then counterclockwise rotation by $\theta$ gives an automorphism $\varphi_{\theta}$ of the cyclic poset $X$. This automorphism will be admissible (Definition \ref{admissible automorphism}) if and only if $\theta<\pi$ since that is the condition under which applying the automorphism twice will not make a complete circle. 
Let $\cC_{\varphi_{\theta}}(X)$ be the stable category of the twisted Frobenius category $\cF_{\varphi_{\theta}}(X)$ as in Definition \ref{twisted stable category}.

The question that we will consider in this section is: under what conditions does the triangulated category $\cC_{\varphi_{\theta}}(X)$  have a cluster structure? We need a collection of clusters $\cS$ closed under mutation where mutations are given by approximations. We will show below that a necessary and sufficient condition for 
$\cC_{\varphi_{\theta}}(X)$ to have a cluster structure is that $\theta=2\pi/N$ for some positive integer $N\ge4$. This statement has already been proven in the case $X=S^1$ in \cite{IT15b}, however here we analyze other subsets of $S^1$, finite and infinite.

Suppose that $\theta=2\pi/N$ with $N\ge4$. 
The cyclic subgroup $\left<\varphi_{\theta}\right>\subset Aut(X)$ generated by $\varphi_{\theta}$ acts freely on $X$.
The orbit 
$O(x)$ of each point $x\in X$ has exactly $N$ points which are equally spaced around the circle $S^1$. 

\begin{lem} \label{Jx} Let $Z_N$ be the standard set of $N$ points in the circle $S^1$ labeled by $\{1,2,\dots,n\}$ and let $\varphi$ be the canonical automorphism of $Z_N$ given by 
$\varphi(i)=i+1$ and $\varphi(n)=1$.
Let $X$ be a subset of $S^1$ which is invariant under rotation in either direction by the angle $\theta$ and let $\varphi_{\theta}\in Aut(X)$ be the  counterclockwise rotation by $\theta$. 
Let  
\[
	J_x: (Z_N,\varphi)\to (X,\varphi_{\theta}) \text{ be defined as } J_x(i)= \varphi_{\theta}^i(x).
\]
\begin{enumerate}
\item Then $J_x$ is a monomorphism of cyclic posets with automorphisms.
\item Then $J_x$ induces full, faithful embedding of Frobenius categories 
$\cJ_x\!:\!\cF_\varphi(Z_N)\!\to\!\cF_{\varphi_{\theta}}(X)$ which  takes projective-injective objects to projective-injective objects.
\item The induced functor on the stable categories $\underline\cJ_x:\cC_\varphi(Z_N)\to \cC_{\varphi_{\theta}}(X)$  is a triangulated full embedding of triangulated categories. 
\end{enumerate}

\end{lem}

\begin{rem}
The triangulated category $\cC_\varphi(Z_N)$ is the standard cluster category of type $A_{N-3}$ which has a cluster structure whose clusters $\cT$ are maximal compatible sets of objects. Since $\varphi$ is standard, compatibility is equivalent to the corresponding geodesics being noncrossing. Therefore, the maximal compatible sets in $\cC_\varphi(Z_N)$ are in bijection with triangulations of the regular $N$-gon. 
\end{rem}
The claim is that these $\underline\cJ_x(\cT)$ (for all $x\in X$) give all of the clusters in $\cC_\theta(X)$.

\begin{prop}\label{clusters for theta}
Let $\theta=2\pi/N$ for $N\in\mathbb N$. 
Let $X$ be a subset of $S^1$ which is invariant under rotation in either direction by the angle $\theta$. 
Then  the triangulated category $\cC_{\varphi_{\theta}}(X)$ has a cluster structure, where clusters are given by the set of all $\underline\cJ_x(\cT)$ where $x\in X$ and $\cT$ is a cluster in $\cC_\varphi(Z_N)$.
\end{prop}

\begin{proof}
 Since $\underline\cJ_x:\cC_\varphi(Z_N)\to \cC_{\varphi_{\theta}}(X)$  is a triangulated full embedding of triangulated categories and $\cC_\varphi(Z_N)$ has a cluster structure, the images under $\underline\cJ_x$ of the clusters in $\cC_\varphi(Z_N)$ give a cluster structure for $\cC_{\varphi_{\theta}}(X)$. This is a very general statement. Let $\cS=\underline\cJ_x(\cT)$ where $\cT$ is a cluster in $\cC_\varphi(Z_N)$. Let $S=\underline\cJ_x(\cT)(T)$ be an object in $\cS$. Then mutation of $\cT$ is $\cT\backslash T\cup T'$ where $T'$ is given by a distinguished triangle $T\to B\to T'\to T[1]$ in $\cC_\varphi(Z_N)$ where $B$ is a left $add(\cT\backslash T)$-approximation of $T$. This maps to a distinguished triangle
 \[
 	S=\underline\cJ_x(T)\to \underline\cJ_x(B)\to S'=\underline\cJ_x(T')\to S[1]
 \]
 in $\cC_{\varphi_{\theta}}(X)$ where $\underline\cJ_x(B)$ is the left $add(\cS\backslash S)$-approximation of $S$. So, the mutation $\cT\backslash T\cup T'$ of $\cT$ maps to the mutation $\cS\backslash S\cup S'$ of $\cS$. So, $\underline\cJ_x$ takes clusters to clusters and commutes with mutation.
Taking the union over all $x\in X$ gives a larger set of clusters which is closed under mutation and gives a cluster structure on $\cC_{\varphi_{\theta}}(X)$
\end{proof}

The converse of Proposition \ref{clusters for theta} is also true:

\begin{thm}\label{thm: all cluster for phi-theta}
The triangulated category $\cC_{\varphi_{\theta}}(X)$ has a cluster structure if and only if $\theta=2\pi/N$ for some positive integer $N\ge4$. When this is the case, all clusters in 
$\cC_{\varphi_{\theta}}(X)$ are given by $\underline\cJ_x(\cT)$ where $x\in X$ and $\cT$ is a cluster in $\cC_\varphi(Z_N)$.
\end{thm}

\begin{proof}
If $\theta=2\pi/N$ then by Proposition \ref{clusters for theta} the  category $\cC_{\varphi_{\theta}}(X)$ has a cluster structure.

Conversely, suppose that $\cC_{\varphi_{\theta}}(X)$ has a cluster structure. Take the triangulated full embedding $\cC_{\varphi_{\theta}}(X)\into \cC_{\varphi_{\theta}}(S^1)$. The cluster structure on $\cC_{\varphi_{\theta}}(X)$ maps to a cluster structure in $\cC_{\varphi_{\theta}}(S^1)$  by the general argument explained in the proof of Proposition \ref{clusters for theta}. By \cite{IT15b} this is possible only if $\theta=2\pi/N$ for some positive integer $N\ge4$ and the clusters in $\cC_{\varphi_{\theta}}(S^1)$ come from the clusters in $\cC_\varphi(Z_N)$. This proves all the statements in the Theorem.
\end{proof}

\begin{rem} Let $\theta=2\pi/N$. Let $X$ be a subset of $S^1$ which is invariant under $\varphi_{\theta}$.
\begin{enumerate}
\item  If $X$ has $kN$ elements then $\cC_{\varphi_{\theta}}(X)$ has $\frac k{N-1}\binom{2N-4}{N-2}$ clusters.
 Since $\cC_\varphi(Z_N)$ has $\frac1{N-1}\binom{2N-4}{N-2}$ clusters. 
 \item  In the example, in Figure \ref{Figure30}, $N=8$ and $k=3$ so $\cC_{\varphi_{\pi/4}}(Z_{24})$ has $\frac 37\binom{12}6= 396$ clusters.
\item Clusters are not maximal $\Ext^1$-compatible sets as in the example in Figure \ref{Figure30}.
\end{enumerate}
\end{rem}

\begin{rem}
 If $X$ is infinite, clusters are in general not reachable one from another. In fact, mutation classes are finite and there are an infinite number of clusters.
\end{rem}

\begin{figure}[htbp]
\begin{center}
\begin{tikzpicture}[scale=.6]
\begin{scope}[xshift=-9cm]
\draw[thick] (0,0) circle [radius=3cm];
\foreach \x in {0,15,30,45,60,75,90}\draw[thick,rotate=\x] (2.9,0)--(3.1,0) (-2.9,0)--(-3.1,0) ;
\foreach \x in {15,30,45,60,75}\draw[thick,rotate=-\x] (2.9,0)--(3.1,0) (-2.9,0)--(-3.1,0) ;
\foreach \x in {0,45,90,135}
\draw[color=blue,very thick,rotate=\x] (2.8,0)--(3.2,0) (-2.8,0)--(-3.2,0) ;
\clip (0,0) circle [radius=3cm];
\cluster{45}{4.24}{3};
\cluster{-45}{4.24}{3};
\cluster{135}{4.24}{3};
\cluster{-135}{4.24}{3};
\draw[very thick] (0,-3)--(0,3);
\end{scope}
\begin{scope}
\draw[thick] (0,0) circle [radius=3cm];
\foreach \x in {0,15,30,45,60,75,90}\draw[thick,rotate=\x] (2.9,0)--(3.1,0) (-2.9,0)--(-3.1,0) ;
\foreach \x in {15,30,45,60,75}\draw[thick,rotate=-\x] (2.9,0)--(3.1,0) (-2.9,0)--(-3.1,0) ;
\clip (0,0) circle [radius=3cm];
\foreach \x in {0,30,60,90,120,150,180,-30,-60,-90,-120,-150}
\bcluster{\x}{3.46}{1.73};
\cluster{45}{4.24}{3};
\cluster{-45}{4.24}{3};
\cluster{135}{4.24}{3};
\cluster{-135}{4.24}{3};
\draw[very thick] (0,-3)--(0,3);
\end{scope}
\end{tikzpicture}
\caption{In the category $\cC_\theta(Z_{24})$ with $\theta=\pi/4$, $J_0(\cC_\varphi(Z_8))$ consists of the 20 objects $E(x,y)$ where $x,y$ are multiples of $\pi/4$ with $|y-x|\ge\pi/2$. This full subcategory contains 132 clusters, one of which is shown on the left. However, these clusters are not maximal $\Ext^1$-compatible sets. For example, 12 objects from $\cC_\theta(Z_{24})$ can be added to this set to form a maximal compatible set with 17 objects as shown on the right.}
\label{Figure30}
\end{center}
\end{figure}
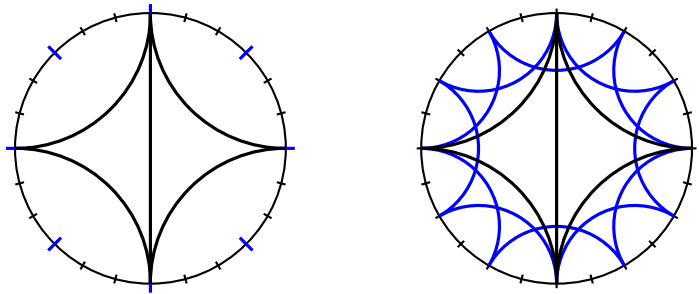
%

\end{document}